\documentclass[12pt]{amsart}
\usepackage{amsmath}
\usepackage{amsthm}
\usepackage{amssymb}
\usepackage{enumerate}
\usepackage{xcolor}
\usepackage{comment}
\newtheorem{theorem}{Theorem}[section]
\newtheorem{lemma}[theorem]{Lemma}
\newtheorem{proposition}[theorem]{Proposition}
\newtheorem{corollary}[theorem]{Corollary}

\newtheorem{remar}[theorem]{Remark}
\theoremstyle{definition}
\newtheorem{examp}[theorem]{Example}
\newtheorem{prob}[theorem]{Open Problem}
\newenvironment{example}{\begin{examp}\rm}{\diams\end{examp}}
\newcommand{\diams}{\unskip\nobreak\hfil\penalty50%
\hskip1em\hbox{}\nobreak\hfil%
$\diamondsuit$\parfillskip=0pt\finalhyphendemerits=0}

\newcommand{\bfind}[1]{\index{#1}{\bf #1}}
\newcommand{\n}{\par\noindent}
\newcommand{\sn}{\par\smallskip\noindent}
\newcommand{\mn}{\par\medskip\noindent}
\newcommand{\bn}{\par\bigskip\noindent}
\newcommand{\pars}{\par\smallskip}
\newcommand{\parm}{\par\medskip}

\newcommand{\cal}{\mathcal}

\newcommand{\sep}{^{\rm sep}}
\newcommand{\chara}{\mbox{\rm char}\,}

\newcommand{\Gal}{\mbox{\rm Gal}\,}

\newcommand{\drk}{\mbox{\rm defrk}\,}
\newcommand{\inddrk}{\mbox{\rm inddrk}\,}
\newcommand{\ann}{\mbox{\rm ann}\,}
\newcommand{\tr}{\mbox{\rm Tr}\,}
\newcommand{\cO}{\mathcal{O}}
\newcommand{\cM}{\mathcal{M}}
\newcommand{\cE}{\mathcal{E}}
\newcommand{\cH}{\mathcal{H}}

\newcommand{\cC}{\mathcal{C}}

\newcommand{\cL}{\mathcal{L}}

\newcommand{\cG}{\mathcal{G}}
\newcommand{\rme}{\mbox{\rm e}\,}

\newcommand{\pH}{\mathop{\raisebox{-.2ex}{\rule[-1pt]{0pt}{1pt}%
\mbox{\large\bf H}}}}

\newcommand{\R}{\mathbb R}
\newcommand{\Q}{\mathbb Q}
\newcommand{\N}{\mathbb N}
\newcommand{\Z}{\mathbb Z}
\newcommand{\F}{\mathbb F}

\begin{document}
\title[Definable coarsenings of valuation rings]{On certain definable coarsenings of valuation rings and their applications}
\author{Franz-Viktor Kuhlmann}
\date{26.\ 11.\ 2025}


\address{Institute of Mathematics, University of Szczecin,
ul. Wielkopolska 15,
70-451 Szczecin, Poland}
\email{fvk@usz.edu.pl}

\begin{abstract}\noindent
We show how suitable extensions $(L|K,v)$ of prime degree of valued fields give rise
to definable coarsenings of the valuation rings of $L$ and $K$. In the case of
Artin-Schreier and Kummer extensions with wild ramification, we can also define the
ramification ideal. We demonstrate the use of the coarsenings on $L$, their maximal
ideals, and the ramification ideals for the classification of defects and for the
presentation of the K\"ahler differentials of the extension of the valuation rings of
$(L|K,v)$, and their annihilators.
Finally, we give a construction that realizes predescribed convex subgroups of suitable
value groups as those that are associated with Galois extensions of degree $p$ with
independent defect, which in turn give rise to definable coarsenings.
\end{abstract}


\maketitle

%
%
\section{Introduction}
In this paper, for Galois extensions $(L|K,v)$ of prime degree, as studied in
\cite{CuKuRz,CuKu}, we will discuss definable coarsenings of the valuation rings of
$L$ and $K$, and their applications to the presentation of the K\"ahler differentials
of the extension of the valuaion rings of $(L|K,v)$. As our main interest are these
applications, we will only deal with definability in suitable expansions of the
language $\cL_{\rm val}$ of valued fields, instead of the language of rings.

Moreover, we will be interested in definable coarsenings of both the valuation ring
$\cO_L$ of $v$ on $L$ and the valuation ring $\cO_K$ of $v$ on $K$; however, it is
the former that are important for our applications. Under certain additional
assumptions the coarsenings of $\cO_K$ have already been shown in \cite{KRS} to be
definable in the ring language.

\pars
The notions and notations we will now use will be introduced in Section~\ref{sectprel}.

%
%
\subsection{Coarsenings defined from immediate elements in valued field extensions}
%
Take any valued field extension $(L|K,v)$ and an arbitrary element $z\in L\setminus K$.
For a nonempty subset $M\subseteq K$ we define
\[
v(z-M)\>:=\>\{v(z-c)\, |\, c\in M\}\>\subseteq\> vL\>.
\]
If $M=K$, then the set $v(z-K)\cap vK$ is an initial segment of $vK$. For the
properties of the sets $v(z-K)$, see \cite[Chapter 2.4]{Ku65}.
If $v(z-K)$ has no maximal element, then we call $z$ an \bfind{immediate element} of
the extension $(L|K,v)$. In this case, $v(z-K)\subseteq vK$.

In Section~\ref{sectv(z-K)}, we will define from an immediate element $z$ a coarsening
$\cO_{z-K}$
of the valuation ring $\cO_L$ of $L$ in the language $\cL_{{\rm val},K}$ of valued
fields with a predicate for membership in $K$. This coarsening plays an important role
in our study of Galois defect extensions of prime degree. If $z$ does not lie in the
completion $\widehat{K}$ of $(K,v)$, then $v(z-K)$ is bounded from above and
$-v(z-K)$ is bounded from below in $vK$. In this case,
\begin{equation}
I_{z-K}\>:=\>\{b\in L\mid \exists c\in K:\; vb\geq-v(z-c)\}
\end{equation}
is a (possibly fractional) $\cO_L$-ideal. When we speak of $\cO_L$-ideals, we always
include fractional ideals, that is, $\cO_L$-modules $I\subset L$ for which there is
some $a\in\cO_L$ such that $aI\subseteq\cO_L\,$.

For an $\cO_L$-ideal $I$, its invariance valuation ring $\cO(I)$ (see the definition
in Section~\ref{sectivr}) is the largest of all coarsenings $\cO'$ of $\cO_L$ such that
$I$ is an $\cO'$-ideal. It is definable in the ring language augmented by a predicate
for membership in $I$. We define $\cO_{z-K}$ to be the
invariance valuation ring of $I_{z-K}\,$.


\mn
%
%
\subsection{Galois defect extensions of prime degree}   \label{Gdepdeg}
These extensions have been studied in \cite{KR} and \cite{CuKuRz}.
Take a valued field $(K,v)$ with $\chara Kv=p>0$, and a Galois defect extension
$\cE=(L|K,v)$ of prime degree $p$. For every
$\sigma$ in its Galois group $\Gal (L|K)$, with $\sigma\ne\,$id, we set
\begin{equation}                        \label{Sigsig}
\Sigma_\sigma\>:=\> \left\{ v\left( \left.\frac{\sigma b-b}{b}\right) \right| \,
b\in L^{\times},\,\sigma b\ne b \right\} \>.
\end{equation}
This set is a final segment of $vK$ and independent of the choice of $\sigma$
(see Theorem~\ref{Sigma_E-H_E}); we denote it by
$\Sigma_\cE\,$. It is shown in \cite[Section~2.4]{KR}
that
\begin{equation}
I_\cE\>=\>(b\in L\mid vb\in\Sigma_\cE)\>=\>\{b\in L\mid vb\in\Sigma_\cE\vee b=0\}
\end{equation}
is the unique ramification ideal of $\cE$. We set $\cO_\cE:=\cO(I_\cE)$ and denote its
maximal ideal by $\cM_\cE\,$. We denote by $\cL_{{\rm val},K}$ the language of valued
fields with a predicate for membership in $K$ and prove in Section~\ref{sectde}:
\begin{proposition}                   \label{eldefIOM}
Take a Galois extension $\cE=(L|K,v)$ of prime degree $p$.
\sn
1) The ideals $\cO_\cE$ and $\cM_\cE$ are $\cL_{{\rm val},K}$-definable in $(L,v)$.
\sn
2) If $\chara K=0$, then assume in addition that $K$ contains a primitive $p$-th root
of unity. Then also the ideal $I_\cE$ is $\cL_{{\rm val},K}$-definable in $(L,v)$.
\end{proposition}

A main aim of this paper is to describe the role the ideals $I_\cE\,$, $\cO_\cE$ and
$\cM_\cE$ play in the description of the structure of Artin-Schreier extensions and
Kummer defect extensions of prime degree. This will be done in Section~\ref{sectde}.

\pars
We say that $\cE$ has \bfind{independent defect} if
\begin{equation}                                     \label{indepdefnew}
I_\cE\>=\>\cM_\cE \quad\mbox{ and $\cM_\cE$ is a nonprincipal $\cO_\cE$-ideal,}
\end{equation}
otherwise we will say that $\cE$ has \bfind{dependent defect}. We will show in
Section~\ref{sectde} that in the case of Artin-Schreier extensions and
Kummer extensions of prime degree, this definition is equivalent to the one given in
\cite{CuKuRz}.

\pars
Let us give an example for the importance of independent defect.
A valued field $(K,v)$ is called a \bfind{roughly deeply ramified field}, or in short
an \bfind{rdr field}, if the following conditions hold:
\sn
{\bf (DRvp)} if $\chara Kv=p>0$, then $vp$ is not the smallest positive element
in the value group $vK$,
\sn
{\bf (DRvr)} if $\chara Kv=p>0$, then $\cO_K/p\cO_K$ is semiperfect if $\chara K=0$,
and the completion $\widehat K$ of $(K,v)$ is perfect if $\chara K=p$.

\pars
The following is a consequence of \cite[Theorem 1.10 1)]{KR}:
\begin{theorem}
Assume that $(K,v)$ is a roughly deeply ramified field. Then every Galois defect
extension $\cE=(L|K,v)$ of prime degree $p=\chara Kv>0$ has independent defect.
\end{theorem}

%
%
\subsection{Deeply ramified fields and K\"ahler differentials}
We call $(K,v)$ a \bfind{deeply ramified field} if it satisfies condition {\bf (DRvr)}
together with
\sn
{\bf (DRvg)} whenever $\Gamma_1\subsetneq\Gamma_2$ are convex subgroups of the value
group $vK$, then $\Gamma_2/\Gamma_1$ is not isomorphic to $\Z$  (that is, no
archimedean component of $vK$ is discrete).
\sn
Every perfect valued field of positive characteristic $p$ and every perfectoid field is
a deeply ramified field with $p$-divisible value group. Every deeply ramified field is
an rdr field.

\pars
A theorem of Gabber and Ramero uses K\"ahler differentials, that is, modules of
relative differentials, to characterize deeply ramified fields
(cf.\ \cite[Theorem~6.6.12 (vi)]{GR} and \cite[Theorem 1.2]{CuKu}). When $A$ is a ring
and $B$ is an $A$-algebra, then we denote by $\Omega_{B|A}$ the K\"ahler differentials
of $B|A$ (see Section~\ref{sectKd}). Given a valued field $(K,v)$, we
denote by $K\sep$ the separable algebraic closure and extend $v$ from $K$ to $K\sep$.
The following result does not depend on the choice of the extension of $v$ since all
of the possible extensions are conjugate.
\begin{theorem}             \label{GRthm}
For a valued field $(K,v)$,
\begin{equation}                         \label{GRdefdr}
\Omega_{\cO_{K\sep}|\cO_K} \>=\> 0
\end{equation}
holds if and only if $(K,v)$ is a deeply ramified field.
\end{theorem}

A main goal of the papers \cite{CuKuRz,CuKu} is to compute the K\"ahler differentials of
Galois extensions $\cE=(L|K,v)$ of prime degree of valued fields and use this to give an alternative
proof of Theorem~\ref{GRthm}. According to \cite[Theorem 1.1]{CuKuRz}, these K\"ahler
differentials can be represented in the form
\begin{equation}                         \label{UV}
\Omega_{\cO_L|\cO_K} \>\simeq\> U/UV
\end{equation}
where $U$ and $V$ are certain $\cO_L$-ideals. Their
computation in the case of defect extensions $\cE$ is dealt with in \cite{CuKuRz} and
we will state the results in Section~\ref{sectde}.

%
%
%


%
%
\subsection{Defectless Galois extensions of prime degree}
The paper \cite{CuKu} is devoted to the case of defectless extensions $\cE$; in
Section~\ref{sectdl} we discuss its results, as well as the
$\cL_{{\rm val},K}$-definition and the role of the valuation
ring $\cO_\cE$ and its maximal ideal $\cM_\cE\,$. The interesting case is the one of
Galois extensions $\cE=(L|K,v)$ of prime degree $q=(vL:vK)$ (which this time is not
necessarily equal to $\chara Kv$). In order to compute the ideals $U$ and $V$
appearing in
(\ref{UV}) we determined in \cite{CuKu} a presentation of $\cO_L$ as a union over a
chain of simple ring extensions of $\cO_K\,$. It depends on a distinction of three
ways in which $vK$ extends to $vL$, and as a byproduct we obtain definitions of
the valuation ideal $\cO_\cE$ and $\cM_\cE\,$. We will show in Section~\ref{sectimp}
that the ideal $\cM_\cE$ is necessary for the presentation of $U$
and $V$, and also for the computation of the annihilator of $\Omega_{\cO_L|\cO_K}$.

In Section~\ref{sectdefI_E} we will present $\cL_{{\rm val},K}$-definitions of the
ramification ideal $I_\cE$ for the defectless wildly ramified case.

%
%
\subsection{Predescribed associated convex subgroups}
For Galois defect extensions $\cE=(L|K,v)$ of prime degree, we have already defined
in Section~\ref{Gdepdeg} the valuation rings $\cO_\cE$. They correspond to convex
subgroups of $vL$ via the definition
\[
H_\cE\>:=\> v\cO_\cE^\times\>=\> v\cO_\cE \cap -v\cO_\cE\>.
\]
Since the extension $\cE$ is immediate, $H_\cE$ is a convex subgroup of both $vK$ and
$vL$.

Using this definition,
we can modify the original definition for independent defect given in
\cite{KR} in the following way: $\cE$ has \bfind{independent defect} if
\begin{equation}                                     \label{indepdef}
\Sigma_{\cE}\>=\> \{\alpha\in vK\mid \alpha >H_\cE\}\>\mbox{ and $vK/H_\cE$ has no smallest positive element;}
\end{equation}
otherwise we will say that $\cE$ has \bfind{dependent defect}. If $(K,v)$ has
rank 1 (i.e., its value group is order isomorphic to a subgroup of $\R$), then
condition (\ref{indepdef}) just means that $\Sigma_{\cE}$ consists of all positive
elements in $vK$. In the case of independent defect, we will call $H_\cE$ the
\bfind{convex subgroup associated with} $\cE$.
In order not to overload our sentences, we will write ``associated convex subgroup''
for ``convex subgroup associated with a Galois defect extension of prime degree''.

For an ordered abelian group $\Gamma$, denote by $\cC(\Gamma)$ the chain of its proper
convex subgroups, and by $\cC_{\rm pr}(\Gamma)$ the chain of its proper principal
convex subgroups. If $H$ is a convex subgroup of $\Gamma$ that is the smallest
among all convex subgroups that contain a given element $\gamma\in\Gamma$, then we
call it a \bfind{principal convex subgroup}, and if it is largest among all convex
subgroups that do not contain a given element $\gamma\in\Gamma$, then we call it a
\bfind{subprincipal convex subgroup}. A subprincipal convex subgroup may
or may not be principal. In Section~\ref{sectpacse} we will prove:
\begin{theorem}                         \label{prescracs}
Let $p$ be a prime and take any totally ordered set $I$. Then there exists an ordered
abelian group $\Gamma$ with $\cC_{\rm pr}(\Gamma)$ order isomorphic to $I$ such that
for any subset $\cC^{\rm sp}\subseteq\cC$ containing only subprincipal convex
subgroups, the following statements hold.
\sn
1) There exists a perfect henselian valued field of characteristic $p$
with value group $\Gamma$ for which the associated convex subgroups are exactly the
elements of $\cC^{\rm sp}$.

\sn
2) Assume in addition that $\Gamma$ has a largest proper convex subgroup. Then there
exists a henselian deeply ramified field of characteristic $0$ and residue
characteristic $p$ with value group $\Gamma$ for which the associated convex
subgroups are exactly the elements of $\cC^{\rm sp}$.
\end{theorem}

In an ordered abelian group with only finitely many proper convex subgroups, each of
them is subprincipal. Therefore, the next result follows immediately from our
theorem:
\begin{corollary}
Let $p$ be a prime and take any finite totally ordered set $I$. Then there exists an
ordered abelian group $\Gamma$ with $\cC_{\rm pr}(\Gamma)$ order isomorphic to $I$
such that for any set $\cH$ of proper convex subgroups of $\Gamma$, there exists a
perfect henselian valued field of characteristic $p$ as well as a henselian deeply
ramified field of characteristic $0$ and residue characteristic $p$
with value group $\Gamma$ for which the associated convex subgroups are exactly the
elements of $\cH$.
\end{corollary}

%
%
\section{Preliminaries}                   \label{sectprel}
%
%
%
\subsection{Notation}                            \label{sectnot}
For a valued field $(K,v)$, we denote the value group by $vK$,  the residue field by
$Kv$,  the valuation ring by $\cO_K\,$, and its maximal ideal by $\cM_K\,$. We set
$vK^{>0}:=\{\alpha\in vK\mid \alpha>0\}$ and $vK^{<0}:=\{\alpha\in vK\mid \alpha<0\}$.
Throughout, we will use the convention that $v0=\infty>\alpha$ for all $\alpha\in vK$.

By $(L|K,v)$ we denote a field extension $L|K$ with a valuation $v$ on $L$, where $K$
is endowed with the restriction of $v$. In this case, there are induced embeddings
of $vK$ in $vL$ and of $Kv$ in $Lv$. The extension $(L|K,v)$ is called
\bfind{immediate} if these embeddings are onto. In this case, if $z\in L\setminus K$,
then $v(z-K)$ has no maximal element, and therefore $z$ is an immediate element of
$(L|K,v)$; this follows from \cite[Lemma 2.9 2)]{Ku65} and the fact that each
subextension of an immediate extension is immediate.

We call $(L|K,v)$ \bfind{unibranched} if the valuation $v$ has only one extension
from $K$ to $L$. A valued field is \bfind{henselian} if and only if all of its
algebraic extensions are unibranched.

If $(L|K,v)$ is a finite unibranched extension, then by the Lemma of Ostrowski
(\cite[Corollary to Theorem 25, Section G, p.\ 78]{ZS}),
\begin{equation}                    \label{feuniq}
[L:K]\>=\> \tilde{p}^{\nu }\cdot(vL:vK)[Lv:Kv]\>,
\end{equation}
where $\nu$ is a non-negative integer and $\tilde{p}$ the
\bfind{characteristic exponent} of $Kv$, that is, $\tilde{p}=\chara Kv$ if it is
positive and $\tilde{p}=1$ otherwise. The factor $d(L|K,v):=\tilde{p}^{\nu }$ is
the \bfind{defect} of the extension $(L|K,v)$.
If $d(L|K,v)=1$, then the extension $(L|K,v)$ is called \bfind{defectless};
otherwise we call it a \bfind{defect extension}. A henselian field $(K,v)$ is a
\bfind{defectless field} if every finite unibranched extension of $(K,v)$ is
defectless; note that this is always the case if $\chara Kv=0$.

%
%
\subsection{Ramification ideals}
If $L|K$ is Galois, then we denote its Galois group by $\Gal L|K$. In this case,
a nontrivial $\cO_L$-ideal contained in $\cM_L$ is called a
\bfind{ramification ideal} of $(L|K,v)$ if it is of the form
\begin{equation}                     \label{genI}
\left(\frac{\sigma b -b}{b}\>\left|\;\> \sigma\in H\,,\>b\in L^\times\right.
\right)
\end{equation}
for some subgroup $H$ of $\Gal L|K$. For more information on ramification ideals, see
\cite{KuTopI}.

%
%
\subsection{K\"ahler differentials}               \label{sectKd}
Assume that $A$ is a ring and $B$ is an $A$-algebra.  Then $\Omega_{B|A}$ denotes the
module of relative differentials  (K\"ahler differentials),  that is, the
$B$-module for which there is a universal derivation
\[
d:\>B\>\rightarrow \>\Omega_{B|A}
\]
such that for every $B$-module $M$  and derivation $\delta: B\rightarrow M$
there is a unique $B$-module homomorphism
\[
\phi:\>\Omega_{B|A}\>\rightarrow\> M
\]
such that $\delta=\phi \circ d$.

\subsection{Invariance group and invariance valuation ring}         \label{sectivr}
\mbox{ }\sn
Take any valued field $(L,v)$ and $\cO_L$-ideal $I$. We set
\begin{equation}                            \label{O(I)}
\cO(I)\>:=\>\{b\in L\mid bI\subseteq I\}
\;\;\;\mbox{\ \ and\ \ }\;\;\;
\cM(I) \>=\>\{b\in L\mid bI\subsetneq I\}\>.
\end{equation}
We call $\cO(I)$ the \bfind{invariance valuation ring} of $I$.
The following is part of \cite[Theorem 3.6]{Kucuts}:
\begin{proposition}                             \label{Ovr}
For every $\cO_L$-ideal $I$, $\,\cO(I)$ is a valuation ring of $L$ containing
$\cO_L\,$, with maximal ideal $\cM(I)$, which is a prime $\cO_L$-ideal. It is the
largest of all valuation rings $\cO$ of $L$ containing $\cO_L$ for which $I$ is an
$\cO$-ideal.
\end{proposition}

If the ideal $I$ is definable in an expansion $\cL$ of $\cL_{\rm val}$, then also
$\cO(I)$ and $\cM(I)$ are $\cL$-definable:
\begin{eqnarray}
\cO(I)&:=&\{b\in L\mid \forall c\in I:\; bc\in I\}, \label{defO(I)}\\
\cM(I)&:=&\{b\in \cO(I)\mid \exists a\in I\> \forall c\in I:\;
bc\ne a\}\>.  \label{defM(I)}
\end{eqnarray}

\parm
For a subset $M$ of an ordered abelian group $\Gamma$, we
define its \bfind{invariance group} to be
\[
\cG(M)\>:=\>\{\gamma\in \Gamma\mid M+\gamma=M\}\>.
\]
This is a subgroup of $\Gamma$, and it is a convex subgroup if $M$ is convex (which in
particular is the case if $M$ is an initial or a final segment of $\Gamma$). If $S$
is a final segment of $\Gamma$ and $\gamma\in\Gamma$, then $\gamma+S:=\{\gamma+\alpha
\mid \alpha\in S\}$ and $-S:=\{-\alpha\mid \alpha\in S\}$ are again final segments of
$\Gamma$ with
\begin{equation}                         \label{G(gamma+S)}
\cG(\gamma+S)\>=\>\cG(S)\>=\>\cG(-S)\>.
\end{equation}
For these facts and more information on invariance groups, see
\cite[Section 2.4]{Kucuts} and \cite{Ku60}.

For every coarsening $\cO$ of $\cO_L\,$, we set
\[
H(\cO)\>:=\> v\cO\cap -v\cO\>=\>v\cO^\times\>.
\]
This is a convex subgroup of the value group $vL$ of $(L,v)$. If $\cM$ is the maximal ideal of $\cO$, then
\begin{equation}                            \label{vM}
v\cM\>=\>\{\alpha\in vL\mid \alpha >v\cO^\times\}\>=\>\{\alpha\in vL\mid \alpha >
H(\cO)\}\>.
\end{equation}
The valuation $w$ associated with $\cO$ is (up to equivalence) given by
\begin{equation}                 \label{wa}
wa\>=\> va/H(\cO)
\end{equation}
for every $a\in K$, the value group of $w$ is canonically isomorphic to $vK/H(\cO)$,
and the value group of the valuation induced by $v$ on the residue field $Kw$ is
canonically isomorphic to $H(\cO)\,$ (cf.\ \cite{ZS}). The function $\cO\mapsto
H(\cO)$ sends every coarsening $\cO$ of $\cO_L$ to a convex subgroup of $vL$. Its
inverse is given by sending a convex subgroup $H$ of $vL$ to
\begin{equation}                          \label{O(H)}
\cO(H)\>:=\> \{b\in K\mid \exists\alpha\in H:\>\alpha\leq vb\}\>.
\end{equation}
We call this the \bfind{coarsening of $\cO_L$ associated with $H$}.

\pars
Further, for every $\cO_v$-ideal $I$ we define
\[
H(I)\>:=\>H(\cO(I))\>. 
\]
By \cite[Theorem 3.6 3)]{Kucuts},
\begin{equation}                                 \label{H(I)}
H(I)\>=\>H(\cO(I))\>=\>\cG(vI)\>
\end{equation}
and
\begin{equation}                                 \label{O(I)2}
\cO(I)\>=\>\cO(\cG(vI))\>=\>\cO(H(I))\>.
\end{equation}

The next result is part of \cite[Lemma 3.5]{Kucuts}.
\begin{lemma}                         \label{H(O)}
For every coarsening $\cO$ of $\cO_L$ with maximal ideal $\cM$,
\begin{equation}                            \label{H=IO=IM}
H(\cO)\>=\>\cG(v\cO)\>=\>\cG(v\cM)\>.
\end{equation}
\end{lemma}

We leave the straightforward proof of the following result to the reader.
\begin{lemma}                         \label{IaI}
If $I$ is an $\cO_L$-ideal and $J=aI$ with $0\ne a\in L$, then $\cO(J)=\cO(I)$,
$\cM(J)=\cM(I)$ and $H(J)=H(I)$.
\end{lemma}

\bn
%
%
\section{Immediate elements in arbitrary valued field extensions}  \label{sectv(z-K)}
%
%
%
%
Take any valued field extension $(L|K,v)$ and $z\in L\setminus K$ an immediate element
in $(L|K,v)$, that is, the set $v(z-K)$ has no maximal element and is an initial
segment of $vK$. We define
\begin{equation}
I_{z-K;K}\>:=\> (b\in K\mid vb\in -v(z-K)\}
\end{equation}
and
\begin{equation}                          \label{I_{z-K}}
I_{z-K}\>:=\> (b\in L\mid \exists c\in K:\> vb\geq -v(z-c)\}\>.
\end{equation}
If $v(z-K)=vK$, then $I_{z-K;K}=K$. If $v(z-K)$ is bounded from above, then $-v(z-K)$
is bounded from below and therefore, $I_{z-K;K}$ is a fractional $\cO_K$-ideal and
$I_{z-K}$ is a fractional $\cO_K$-ideal. We set $\cO_{z-K;K}:=\cO(I_{z-K;K})$ (taken
in $(K,v)$), and denote its maximal ideal by $\cM_{z-K;K}\,$. Likewise, we set
$\cO_{z-K}:=\cO(I_{z-K})$ (taken in $(L,v)$), and denote its maximal ideal by
$\cM_{z-K}\,$. We see that by (\ref{I_{z-K}}),
$I_{z-K}$ is definable in $L$ in the language $\cL_{{\rm val},K}$ with parameter $z$.
Hence by (\ref{defO(I)}) and  (\ref{defM(I)}), also the invariance valuation ring
$\cO_{z-K}$ and its maximal ideal $\cM_{z-K}$ are $\cL_{{\rm val},K}$-definable in
$L$ with parameter $z$.

Since $z$ is not a parameter in $K$, we may in general not have an elementary
definition of $\cO_{z-K;K}$ and $I_{z-K;K}$ in $(K,v)$. For example, we have not even
excluded the case that $z$ is transcendental over $K$. On the other hand, if $z$ is
algebraic over $K$ with a suitable minimal polynomial, then the situation may change,
as we will see in Sections~\ref{sectdee} and~\ref{sectdem}.

\pars
Now we define
\begin{equation}
H_{z-K;K}\>:=\> H(\cO_{z-K;K})\>=\>H(\cO(I_{z-K;K}))
\end{equation}
and
\begin{equation}                         
H_{z-K}\>:=\> H(\cO_{z-K})\>=\>H(\cO(I_{z-K}))\>.
\end{equation}

By (\ref{H(I)}) and (\ref{G(gamma+S)}),
\[
H_{z-K;K}\>=\> \cG(vI_{z-K;K})\>=\> \cG(-v(z-K)) \>=\> \cG(v(z-K))\>.
\]
We observe that $H_{z-K;K}$ is a proper convex subgroup of $vK$ if and only if $v(z-K)$
is bounded from above, and that this holds if and only if $z$ does not lie in the
completion of $(K,v)$. If $H_{z-K;K}$ is not a proper convex subgroup, that is,
$H_{z-K;K}=vK$, then $\cO_{z-K;K}=K$, i.e., the corresponding valuation is trivial.
Otherwise, this coarsening of $\cO_K$ is nontrivial.

\pars
Now assume in addition that the extension $(L|K,v)$ is immediate. Then as mentioned in
Section~\ref{sectnot}, every $z\in L\setminus K$ is an immediate element in $(L|K,v)$,
and $v(z-K)$ is an initial segment of $vL=vK$. In this case, $H_{z-K;K}$
is also a convex subgroup in $vL$, and moreover,
\[
vI_{z-K}\>=\> \{\alpha\in vL\mid \exists c\in K:\> \alpha\geq -v(z-c)\} \>=\>
-v(z-K) \>=\> vI_{z-K;K}\>.
\]
Using this together with (\ref{H(I)}), we obtain:
\[
H_{z-K}\>=\> H(\cO(I_{z-K})) \>=\> \cG(vI_{z-K}) \>=\>  \cG(vI_{z-K;K}) \>=\>
H(\cO(I_{z-K;K})) \>=\> H_{z-K;K}\>.
\]

\bn
%
%
\section{Defect extensions of prime degree}                    \label{sectde}
%

%
%
%

Take a Galois defect extension $\cE=(L|K,v)$ of prime degree $p$. We set
\[
H_\cE\>:=\> H(\cO_\cE)\>=\>H(\cO(I_\cE))\>.
\]
By Lemma~\ref{H(O)}, $H_\cE$ is the invariance group of $v\cO_\cE$ and of $v\cM_\cE\,$.
By (\ref{H(I)}), Lemma~\ref{H(O)} and the definition of $I_\cE\,$,
\begin{equation}                               \label{H_E}
H_\cE\>=\> \cG(vI_\cE)\>=\>\cG(\Sigma_\cE)\>=\>\cG(v\cO_\cE)\>=\>\cG(v\cM_\cE)\>.
\end{equation}

\begin{theorem}                             \label{Sigma_E-H_E}
For every Galois defect extension $\cE=(L|K,v)$ of prime degree $p$, the following
statements hold.
\sn
1) The set $\Sigma_\sigma$ is a final segment of $vK^{>0}$ and independent of the
choice of a generator $\sigma$ of $\Gal L|K$.
\sn
2) For every $a\in L\setminus K$ and every generator $\sigma$ of $\Gal L|K$,
\begin{equation}                                \label{S_E}
\Sigma_\cE\>=\>-v(a-K)+v(\sigma a-a)
\end{equation}
and
\begin{equation}                                \label{S_Eapp}
I_\cE\>=\> (\sigma a-a)I_{a-K},\;\; H_\cE\>=\>H_{a-K},\;\; \cO_\cE\>=\>\cO_{a-K},
\;\;\mbox{and }\;\cM_\cE\>=\>\cM_{a-K}\>.
\end{equation}
\sn
3) For every $a\in L\setminus K$,
\begin{equation}                                \label{H_EGG}
H_\cE\>=\>\cG(v(a-K))\>=\>\cG(-v(a-K))\>.
\end{equation}
\end{theorem}
\begin{proof}
1): By \cite[Theorem~3.5]{KR}, $\Sigma_\sigma$ is independent of the choice
of a generator $\sigma$ of $\Gal L|K$; so we denote it by $\Sigma_\cE\,$. By
\cite[Theorem~3.4]{KR}, $\Sigma_\cE$ is a final segment of $vK^{>0}$. Note that
$vK^{>0}=vL^{>0}$ since $vK=vL$, as the extension $(L|K,v)$ is immediate.
\sn
2): Equation~(\ref{S_E}) is part of \cite[Theorem~3.4]{KR}. It implies
Equation~(\ref{S_Eapp}) by way of the definitions of $I_\cE$ and $I_{a-K}$,
and Lemma~\ref{IaI}.
\sn
3): From (\ref{H_E}) we know that $H_\cE$ is equal to $\cG(\Sigma_\cE)$, and by
(\ref{S_E}) this is equal to $\cG(-v(a-K)+v(\sigma a-a))$. Since $-v(a-K)+
v(\sigma a-a)$ is a final segment of $vK$ by part 1) of our theorem and $\alpha:=
v(\sigma a-a)\in vL$,
we can infer from equation (\ref{G(gamma+S)}) that $\cG(-v(a-K)+v(\sigma a-a))
=\cG(-v(a-K))$. Finally, the equality $\cG(v(a-K))=\cG(-v(a-K))$ follows from
\cite[Lemma 2.12 3)]{Kucuts}.
\end{proof}

Based on part 2) of this theorem, we can now give the
\sn
{\it Proof of part 1) of Proposition~\ref{eldefIOM}}:\n
We have $\cO_\cE=\cO(I_{a-K})=\{b\in L
\mid bI_{a-K}\subseteq I_{a-K}\}=\{b\in L\mid \forall c\in K\;\exists c'\in K:\;
vb(a-c)=v(a-c')\}\cup\{0\}$, where we use that $v(a-K)$ is a final segment of $vL$.
Since $\cO(I_{a-K})$ does not depend on the choice of $a\in L\setminus K$, $\cO_\cE$
has the following parameter free definitions in the language $\cL_{{\rm val},K}\,$:
\begin{equation}
\cO_\cE\>=\>\{b\in L \mid\forall x\in L\setminus K\;\forall c\in K\;\exists c'\in K:\;
vb(x-c)=v(x-c')\}\cup\{0\}\>,
\end{equation}
and the quantifier ``$\forall x\in L\setminus K$'' can also be replaced by
``$\exists x\in L\setminus K$''.
Further, $\cM_\cE= \{b\in L\mid bI_{a-K}\subsetneq I_{a-K}\}$ has the following
parameter free definition in the language $\cL_{{\rm val},K}\,$:
\begin{equation}
\cM_\cE\>=\>\{b\in L \mid b\in\cO_\cE \wedge \exists c\in K\; \forall c'\in K:\;
vb(x-c)\ne v(x-c')\}\>.
\end{equation}

\parm
Let us show that our definitions (\ref{indepdefnew}) and (\ref{indepdef})
of independent defect are equivalent.
By definition of $I_\cE$ we have $\Sigma_\cE=vI_\cE\,$. By (\ref{vM}), $v\cM_\cE=
\{\alpha\in vK\mid \alpha >H(\cO_\cE)\}=\{\alpha\in vK\mid \alpha >H_\cE\}$. Hence
(\ref{indepdef}) reads as $vI_\cE=v\cM_\cE\,$. Since the function $M\mapsto vM:=
\{va\mid a\in M\}$ that sends every $\cO_L$-module $M\subseteq L$ to a corresponding
final segment in $vL$ is bijective, the latter equality is equivalent to the equality
$I_\cE=\cM_\cE\,$. Further, as $vK/H_\cE$ is the value group of $\cO_\cE\,$,
$vK/H_\cE$ having no smallest positive element is equivalent to $\cM_\cE$ being a
nonprincipal $\cO_\cE$-module.

Finally, we show that (\ref{indepdef}) is equivalent to the condition (6) in the
original definition of independent defect in \cite{KR}. It is obvious that
(\ref{indepdef}) implies the latter. For the converse, assume that $\Sigma_{\cE}=
\{\alpha\in vK\mid \alpha >H\}$ for some proper convex subgroup $H$ of $vL$ such that
$vL/H$ has no smallest positive element. Then by \cite[Lemma 2.13 5)]{Kucuts},
$H$ is the invariance group of $\Sigma_{\cE}$, hence by (\ref{H_E}), it is equal to
$H_\cE\,$.

\parm
While we have given elementary definitions of $\cO_\cE$ and $\cM_\cE\,$, the problem
with doing the same for $I_\cE$ is that we may not have enough information on the
factor $\sigma a-a$. We will now show that this changes when we know that the extension
is an Artin-Schreier or a Kummer extension of prime degree. We will thereby prove
part 2) of Proposition~\ref{eldefIOM}.

%
%
\subsection{The equal characteristic case}                    \label{sectdee}
Let us first discuss the case where $(K,v)$ is of equal positive characteristic, that
is, $\chara K= \chara Kv=p>0$. Then every Galois defect extension $\cE=(L|K,v)$ of
prime degree $p$ is an Artin-Schreier extension, that is, generated by an
\bfind{Artin-Schreier generator} $\vartheta\in L\setminus K$ with $\vartheta^p
-\vartheta\in K$. By \cite[Theorem 3.5]{KR},
\begin{equation}                                \label{S_E=AS}
\Sigma_\cE\>=\>-v(\vartheta-K)\>.
\end{equation}
for every such $\vartheta$. Further, $v(\sigma\vartheta-\vartheta)=0$, hence
\[
I_\cE\>=\>I_{\vartheta-K}\>=\> \{b\in L\mid \exists c\in K:\> vb\geq -v(\vartheta-c)\}
\]
in this case. Equation (\ref{S_E=AS}) shows that the set $v(\vartheta-K)$ does not
depend on the choice of the Artin-Schreier generator of $L|K$, hence $I_\cE$ has the
following parameter free definitions in the language $\cL_{{\rm val},K}\,$:
\begin{equation}                     \label{defIE1}
I_\cE\>=\>\{b\in L \mid\exists x\in L\setminus K \;\exists c\in K:\> x^p-x\in K
\>\wedge\> vb\geq -v(\vartheta-c)\}\>
\end{equation}
and
\begin{equation}                     \label{defIE2}
I_\cE\>=\>\{b\in L \mid\forall x\in L\setminus K \;\exists c\in K:\> x^p-x\in K
\>\rightarrow\> vb\geq -v(\vartheta-c)\}\>.
\end{equation}

\pars
Now assume that $\cE$ has independent defect with associated convex subgroup $H_\cE\,$.
By \cite[Theorem 1.7]{CuKuRz}, this holds if and only if
\begin{equation}                    \label{eq1}
v(\vartheta^p-\vartheta-\wp(K))\>=\>\{\alpha\in pvK\mid \alpha< H_\cE\}\>.
\end{equation}
Since $vL/H_\cE$ has no smallest positive element, equation~(\ref{eq1}) is equivalent
to
\begin{equation}
H_\cE\>=\>\{\beta\in vK\mid \beta>v(\vartheta^p-\vartheta-\wp(K))\mbox{ and }
-\beta>v(\vartheta^p-\vartheta-\wp(K))\}\>.
\end{equation}
In other words,
\begin{equation}
H_\cE\>=\>\{\pm\beta\in vK\mid \forall c\in K: v(\vartheta^p-\vartheta-c^p+c)
<\beta\leq 0\}\>.
\end{equation}

The convex subgroup $H_\cE$ gives rise to an $\cL_{\rm val}$-definition of the
coarsening $\cO_{\cE;K}=\cO(H_\cE)=\{b\in K\mid \exists\alpha\in H_\cE:\>\alpha\leq
vb\}$ (taken in $K$) of the valuation ring $\cO_K\,$, namely
\begin{equation}                               \label{defO_EAS}
\cO_{\cE;K}\>=\> \{b\in K\mid \forall c\in K: v(\vartheta^p-\vartheta-c^p+c)< vb\}\>,
\end{equation}
whose value group is $vK/H_\cE\,$.

\pars
For our applications, we are more interested in the coarsening of $\cO_L$ corresponding
to $H_\cE$. By (\ref{O(H)}),
\begin{equation}                           \label{O_E}
\cO_\cE\>=\>\cO(H_\cE)\>=\>\{b\in L\mid \exists\alpha\in H_\cE:\>\alpha\leq vb\}\>.
\end{equation}
By the definition of independent defect combined with Equation~(\ref{S_E=AS}), if
$\cE$ has independent defect with associated convex subgroup $H_\cE\,$, then
\begin{equation}                    \label{eq2}
v(\vartheta-K)\>=\>\{\alpha\in vK\mid \alpha< H_\cE\}\>.
\end{equation}
%
%
Since this does not depend on the choice of the Artin-Schreier generator of $L|K$,
$\cO_\cE$ has the following parameter free definitions in the language
$\cL_{{\rm val},K}\,$:
\begin{equation}
\cO_\cE\>=\>\{b\in L \mid\forall x\in L\setminus K \;\;\forall c\in K:\;  x^p-x\in
K\>\rightarrow\> v(x-c)< vb\}\>
\end{equation}
and
\begin{equation}
\cO_\cE\>=\>\{b\in L \mid\exists x\in L\setminus K \;\;\forall c\in K:\;  x^p-x\in
K\>\wedge\> v(x-c)< vb\}\>.
\end{equation}
Also the maximal ideal $\cM_\cE$ of $\cO_\cE$ has a
parameter free definition in the language $\cL_{{\rm val},K}\,$:
\begin{equation}
\cM_\cE\>=\>
\{b\in L\mid  \exists x\in L\setminus K\; \exists c\in K:\; x^p-x\in K\>\wedge\>
-v(x-c) \leq vb\}\>.
\end{equation}

\mn
%
%
\subsection{The mixed characteristic case}                    \label{sectdem}
Now we discuss the case where $(K,v)$ is of mixed characteristic, that is, $\chara K=0$
and $\chara Kv=p>0$. We assume in addition that $K$ contains a primitive $p$-th root of
unity $\zeta_p\,$. Then every Galois defect extension $\cE=(L|K,v)$ of prime
degree $p$ is a \bfind{Kummer extension}, that is, generated by a \bfind{Kummer
generator} $\eta\in L\setminus K$ with $\eta^p\in K$. Then by \cite[Theorem 3.5]{KR}
and \cite[Lemma 2.5]{KR},
\begin{equation}                                \label{S_E=K}
\Sigma_\cE\>=\>v(\zeta_p -1)-v(\eta-K)\>=\>\frac{1}{p-1}vp-v(\eta-K)\>
\end{equation}
for every such $\eta$. Further, $\sigma\eta-\eta=\zeta_p -1$ for suitable $\zeta_p\,$, hence
\[
I_\cE\>=\>(\zeta_p -1)I_{\vartheta-K}\>=\> \{b\in L\mid \exists c\in K:\> vb\geq
\frac{1}{p-1}vp-v(\vartheta-c)\}
\]
in this case. Similarly as in the equal characteristic case, $I_\cE$ has the
following parameter free definitions in the language $\cL_{{\rm val},K}\,$:
\begin{equation}                     \label{defIE3}
I_\cE\>=\>\{b\in L \mid\exists x\in L\setminus K \;\exists c\in K:\> x^p-x\in K
\>\wedge\> vb\geq \frac{1}{p-1}vp-v(x-c)\}\>
\end{equation}
and
\begin{equation}                     \label{defIE4}
I_\cE\>=\>\{b\in L \mid\forall x\in L\setminus K \;\exists c\in K:\> x^p-x\in K
\>\rightarrow\> vb\geq \frac{1}{p-1}vp-v(x-c)\}\>.
\end{equation}

\pars
Now assume that $\cE$ has independent defect with associated convex subgroup $H_\cE\,$.
By \cite[Theorem 1.7]{CuKuRz}, this holds if and only if
\begin{equation}                    \label{eq1K}
v(\eta^p-K^p)   
\>=\>v(\zeta_p -1)^p+\{\alpha\in pvK\mid \alpha< H_\cE\}\>.
\end{equation}

Similarly as in the equal characteristic case, we obtain that
\begin{equation}
H_\cE\>=\>\{\pm\beta\in vK\mid \forall c\in K: v(\eta^p-c^p)-\frac{p}{p-1}vp
<\beta\leq 0\}\>,
\end{equation}
and we have the $\cL_{\rm val}$-definition
\begin{equation}                               \label{defO_EK}
\cO_{\cE;K}\>:=\>   
\{b\in K\mid \forall c\in K: v(\eta^p-c^p)-\frac{p}{p-1}vp < vb\}\>.
\end{equation}
Note that for this definition and definition (\ref{defO_EAS}) it is not needed that
$(K,v)$ be henselian, and that in fact, they will be applied to deeply ramified fields,
which are not required to be henselian. For the case of henselian fields $(K,v)$, these
definitions are used in \cite[Theorem 4.11]{KRS} to define corresponding henselian
valuations on $K$ that are definable in the language of rings.

\parm
By the definition of independent defect combined with Equation~(\ref{S_E=K}), if
$\cE$ has independent defect with associated convex subgroup $H_\cE\,$, then
\begin{equation}                    \label{eq2K}
v(\vartheta-K)-v(\zeta_p-1)\>=\>\{\alpha\in vK\mid \alpha< H_\cE\}\>.
\end{equation}
Hence in this case, using again (\ref{O_E}) together with the fact that equation~(\ref{eq2K}) is independent of the choice of $\eta\in L\setminus K$
satisfying $\eta^p\in K$), we can give the following parameter free
$\cL_{{\rm val},K}$-definitions of $\cO_\cE$:
\begin{eqnarray*}
\cO_\cE\>=\>\{b\in L &|& \forall x\in L\setminus K\;\; \forall c\in K:\>
(x^p\in K\,\wedge\,v(x-1)>0)\\
&& \rightarrow\> v(x-c)-\frac{1}{p-1}vp< vb\}\>
\end{eqnarray*}
and
\begin{eqnarray*}
\cO_\cE\>=\>\{b\in L &|& \exists x\in L\setminus K: x^p\in K\,\wedge\,v(x-1)>0\\
&& \wedge\;\forall c\in K: v(x-c)-\frac{1}{p-1}vp< vb\}\>.
\end{eqnarray*}
Also the maximal ideal of $\cO_\cE$ admits a parameterfree
$\cL_{{\rm val},K}$-definition:
\begin{eqnarray*}
\cM_\cE\>=\>\{b\in L &|& \exists x\in L\setminus K: x^p\in K\,\wedge\,v(x-1)>0\\
&& \wedge \;\exists c\in K:
-v(x-c)+\frac{1}{p-1}vp\leq vb\}\>.
\end{eqnarray*}

\mn
%
%
\subsection{Properties and applications of $I_\cE\,$, $\cO_\cE$ and $\cM_\cE$}
We keep our assumption that $\cE=(L|K,v)$ is a Galois defect extension of prime
degree $p$.
\pars
Equations (\ref{defIE1}), (\ref{defIE2}), (\ref{defIE3}) and (\ref{defIE4})
prove part 2) of Theorem~\ref{eldefIOM}.
\pars
The following facts are proven in \cite{CuKuRz}. Part 1) follows directly from our
definition $\cO_\cE=\cO(I_\cE)$ in the introduction together with Proposition~\ref{Ovr}
which implies the assertion. However, in \cite{CuKuRz}, under the additional assumption
that $K$ contains a primitive $p$-th root of unity if $\chara K=0$, $\cO_\cE$ is
defined in a different way, and our assertion is part of \cite[Theorem 1.4]{CuKuRz},
as is part 2).
\begin{proposition}
Take a Galois extension $\cE=(L|K,v)$ of prime degree $p$ with independent defect.
\sn
1) \ The ideal $\cM_\cE$ is equal to the ramification ideal $I_\cE$,  and
$\cO_\cE$ is the largest of all coarsenings $\cO'$ of $\cO_L$  such that $I_\cE$ is
an $\cO'$-ideal.
\sn
2) \ If $\chara K=0$, then assume in addition that $K$ contains a primitive $p$-th root
of unity. Then the trace $\tr_{L|K} \left(\cM_L\right)$ is equal to $\cM_\cE\cap K$.
\end{proposition}

\parm
The valuation ring $\cO_\cE$ is
of interest for the computation of the annihilator of $\Omega_{\cO_L|\cO_K}$.
The annihilator of an $\cO_L$-module $M$ is the largest among all $\cO_L$-ideals $J$
for which $JM=\{0\}$; we denote it by $\ann M$. From \cite[Theorem 1.4]{CuKuRz} we
know that
\[
\Omega_{\cO_L|\cO_K}\>\cong\>  I_\cE/I_\cE^p\>,
\]
which is zero if and only if $\cE$ has independent defect; in this case, $\ann
\Omega_{\cO_L|\cO_K}=\cO_L\,$. For the case of dependent defect, we infer from
\cite[Proposition 4.7 2)]{CuKuRz}, denoting by $v_\cE$ the valuation on $L$ having
valuation ring $\cO_\cE\,$:
\begin{proposition}
If there is $a\in K$ such that $v_\cE I_\cE^{p-1}$ has infimum $v_\cE a$ in
$v_\cE L$ but does not contain this infimum, then
\begin{equation}             
\ann \Omega_{\cO_L|\cO_K} \>=\> a\,\cO(I_\cE) \>,
\end{equation}
which properly contains $I_\cE^{p-1}$. In all other cases, $\ann \Omega_{\cO_L|\cO_K}=
I_\cE^{p-1}$.
\end{proposition}

%
%
\section{Defectless extensions}                    \label{sectdl}
We take an extension $\cE=(L|K,v)$ of prime degree $q$, not necessarily equal to
$\chara Kv$. Then either $[L:K]=(vL:vK)$ or $[L:K]=[Lv:Kv]$. {\it We will discuss
the more interesting case of $[L:K]=(vL:vK)$, which we will assume throughout.}

\pars
We define $H_\cE$ to be the largest convex subgroup of $vL$ which is also a convex
subgroup of $vK$; it exists since unions over arbitrary collections of convex
subgroups are again convex subgroups. We take $\cO_\cE$ to be the coarsening
$\cO(H_\cE)$ of $\cO_L$ associated with $H_\cE$ so that its value group is
$vL/H_\cE\,$, and denote its maximal ideal by $\cM_\cE\,$.

The subgroup $H_\cE$ defined here has important
similarities with the convex subgroup $H_\cE$ defined in the defect case.

\pars
We distinguish three mutually exclusive cases describing how $vK$ extends to $vL$; for
convenience, we use the notation of \cite{CuKu}:
\sn
(DL2a): there is no smallest convex subgroup of $vL$ that properly contains $H_\cE\,$;
\sn
(DL2b): there is a smallest convex subgroup $\tilde H_\cE$ of $vL$ that properly contains
$H_\cE\,$, and the archimedean quotient $\tilde H_\cE/H_\cE$ is dense;
\sn
(DL2c): there is a smallest convex subgroup $\tilde H_\cE$ of $vL$ that properly contains
$H_\cE\,$, and the archimedean quotient $\tilde H_\cE/H_\cE$ is discrete.

\pars
Our goal is to find an element $x\in L$ with $vx\notin vK$ such that
\begin{equation}                      \label{(DL2)}
\cO_L\>=\> \bigcup_{c\in K\mbox{\tiny\ with } vcx>0} \cO_K[cx] \>.
\end{equation}
If $c,c'\in K$ with $vc\geq vc'\,$, then $cx=\frac{c}{c'}c' x\in
\cO_K[c' x]$, hence $\cO_K[c x] \subseteq \cO_K[c' x]$.
%
\begin{theorem}{\cite[Theorem 3.3]{CuKu}}                       \label{thmgen}
Take an extension $\cE=(L|K,v)$ of prime degree $q=(vL:vK)$, with $x_0\in L$ such that
$vx_0\notin vK$. Then $qvx_0\in vK$, and the following assertions hold.
\sn
1) If $\cE$ is of type (DL2a) or (DL2b), then (\ref{(DL2)}) holds for $x=x_0\,$.
\sn
2) If $\cE$ is of type (DL2c), then (\ref{(DL2)}) holds for
$x=x_0^j$ with suitable $j\in\{1,\ldots,q-1\}$. If in addition $H_\cE=\{0\}$,
then $\cO_L=\cO_K[cx]$ for suitable $c\in K$.

\sn
The assumption of part 1) holds in particular when every archimedean component of $vK$
is dense, and this in turn holds for every deeply ramified field $(K,v)$.
\end{theorem}

With $x$ as in this theorem, we have:
\begin{proposition}{\cite[Proposition 3.4]{CuKu}}             \label{propgen}
The $\cO_L$-ideal $\cM_\cE$ is equal to the $\cO_L$-ideal
\begin{equation}                 \label{Ixdef}
I_x\>:=\> (cx\mid c\in K \mbox{ with } vcx>0)\>.
\end{equation}
\end{proposition}

\begin{corollary}
The set $\{vcx\mid c\in K \mbox{ with } vcx>0\}$ is coinitial in $vK^{>0}\setminus
H_\cE\,$.
\end{corollary}

\begin{lemma}
For every $x_0$ with $vx_0\notin vK$ we have $I_{x_0}\subseteq I_x\,$.
\end{lemma}
\begin{proof}
Take $c_0\in K$ such that $vc_0x_0>0$, so that $c_0x_0\in\cO_L\,$. If $x$ is as in
Theorem~\ref{thmgen}, then there is $c\in K$ such that $c_0x_0\in\cO_K[cx]$.
Consequently, $c_0x_0\in I_x\,$. This proves that $I_{x_0}\subseteq I_x\,$.
\end{proof}

From this lemma together with Proposition~\ref{propgen} we obtain the following
parameter free $\cL_{{\rm val},K}\,$-definition of $\cM_\cE$:
\begin{equation}
\cM_\cE \>=\> \{b\in L\mid \exists x\in L\setminus K: (\forall y\in K: \>vx\ne vy)
\>\wedge\> \exists c\in K: \; va\geq vcx>0\}\>.
\end{equation}

From this, we can define $\cO_\cE$ by including the units of $\cO_\cE$:
\[
\cO_\cE \>=\> \{b\in L\mid \forall x\in \cM_\cE:\; -vx< vb\} \>.
\]

\mn
%
%
\subsection{The ramification ideal}                \label{sectdefI_E}
Take a unibranched defectless Galois extension $\cE=(L|K,v)$ of prime degree
$p=(vL:vK)=\chara Kv$. We denote by $I_\cE$ the ramification ideal of $\cE$.
From \cite[Theorem 3.15]{KuTopI} we obtain:
\begin{theorem}                              \label{ThmASKgen}
1) If $\cE$ is an Artin-Schreier extension, then it admits an Artin-Schreier
generator $\vartheta$ of value $v\vartheta\leq 0$ such that $v\vartheta\notin vK$. For
every such $\vartheta$,
\begin{equation}
I_{\cE}\>=\> \left(\frac{1}{\vartheta} \right)\>.
\end{equation}
\sn
2) Let $\cE$ be a Kummer extension. Then there are two cases:
\sn
a) \ $\cE$ admits a Kummer generator $\eta$ such that $0< v\eta\notin vK$. For
every such $\eta$,
\begin{equation}
I_{\cE}\>=\> (\zeta_p-1)\>.
\end{equation}
\sn
b) \ $\cE$ admits a Kummer generator $\eta$ such that $\eta$ is a $1$-unit
with $v(\zeta_p-1)\geq v(\eta-1)\notin vK$. For every such $\eta$,
\begin{equation}                   \label{Icase2b}
I_{\cE}\>=\> \left(\frac{\zeta_p-1}{\eta-1}\right)\>.
\end{equation}
\end{theorem}

Let us show that under the assumptions of the theorem, $I_\cE$ always has a
parameter free $\cL_{{\rm val},K}\,$-definition. If $\cE$ is an Artin-Schreier
extension, then we can define
\begin{eqnarray*}
I_{\cE}\>:=\> \{\,b\in L &|& \exists x\in L:\; x^p-x\in K \>\wedge\>
vx\leq 0\\
&&\wedge\> (\forall y\in K: \>vx\ne vy) \>\wedge\> vb\geq vx\,\}\>.
\end{eqnarray*}

If $\cE$ is a Kummer extension, then in case 2)a) of the theorem, we have $v\eta>0$
and therefore, $v(\eta-1)=0$. Thus, we can also in this case use (\ref{Icase2b}) for
the definition of $I_\cE\,$:

\begin{eqnarray*}
I_{\cE}\>:=\> \{\,b\in L &|& \exists x\in L:\; x^p\in K \>\wedge\>
vp\geq (p-1)v(x-1) \\
&&\wedge\> (\forall y\in K: \> 0<vx\ne vy\,\vee\,0<v(x-1)\ne vy) \\
&&\wedge\> (p-1)vb\geq vp-(p-1)v(x-1)\,\}\>.
\end{eqnarray*}

\mn
%
%
\subsection{Importance of the ideals $I_\cE\,$, $\cO_\cE$ and $\cM_\cE$}\label{sectimp}
We will now summarize the results for defectless Galois extensions $\cE=(L|K,v)$ which
will demonstrate the importance of the ideals $I_\cE\,$, $\cO_\cE$ and $\cM_\cE\,$. If
$[L:K]=q\ne\chara K$, then we will assume that $K$ contains a $q$-th root of unity.

\begin{theorem}{\cite[Theorem 4.6]{CuKu}}                       \label{OASe}
Take an Artin-Schreier extension $\cE=(L|K,v)$ of degree $p=(vL:vK)$. Then
\begin{equation}                \label{Ome=p}
\Omega_{\cO_L|\cO_K}\>\cong\>  I_\cE\cM_\cE/(I_\cE\cM_\cE)^p
\end{equation}
as $\cO_L$-modules; in particular, $\Omega_{\cO_L|\cO_K}\ne 0$.
\end{theorem}

The following is a reformulation of \cite[Theorem 4.6]{CuKu}.
\begin{theorem}                                     \label{OKume}
Let $\cE=(L|K,v)$ be a Kummer extension of prime degree $q$ with $\rme(L|K)=q$.

If $q\ne\chara Kv$, then
\begin{equation}                                     \label{OKumec1}
\Omega_{\cO_L|\cO_K}\>\cong\> \cM_\cE/\cM_\cE^q \>
\end{equation}
as $\cO_L$-modules.

If $q=\chara Kv$, then
\begin{equation}                                     \label{OKumecp}
\Omega_{\cO_L|\cO_K}\>\cong\> I_\cE\cM_\cE/(I_\cE\cM_\cE)^q\>
\end{equation}
as $\cO_L$-modules.

In case 2)a) of Theorem~\ref{ThmASKgen}, we have that $\Omega_{\cO_L|\cO_K}= 0$ if and
only if $q\notin\cM_\cE$ and $\cM_\cE$ is a nonprincipal $\cO_\cE$-ideal. The
condition $q\notin\cM_\cE$ always holds when $q\ne\chara Kv$.

In case 2)b) of Theorem~\ref{ThmASKgen}, we always have that
$\Omega_{\cO_L|\cO_K}\ne 0$.
\end{theorem}

Let us compute the annihilators of $\Omega_{\cO_L|\cO_K}$ in the above cases whenever
it is nonzero. The following is Poposition 3.21 of \cite{Kucuts}, adapted to our current
notation.
\begin{proposition}                        \label{ann}
Take $n\geq 2$, $a\in\cO_L$ and $\cO$ a valuation ring containing $\cO_L$ with maximal
ideal $\cM$. Assume that $(a\cM)^n\ne a\cM$. Then the following statements hold.
\sn
1) We have that
\[
\ann a\cM/(a\cM)^n \>=\> \left\{
\begin{array}{ll}
(a\cM)^{n-1} & \mbox{if $\cM$ is a principal $\cO$-ideal, } \\
(a\cO)^{n-1} = a^{n-1}\cO & \mbox{if $\cM$ is a nonprincipal $\cO$-ideal.}
\end{array}\right.
\]
2) \ The annihilator is equal to $\cM_L$ if and only if $n=2$, $a\notin\cM_L=\cM$ and
$\cM_L$ is a principal $\cO_L$-ideal.
\end{proposition}

Since $I_\cE$ is a principal $\cO_L$-ideal, we can choose $a\in \cO_L$ such
that $I_\cE=a\cO_L$ to obtain that
\[
I_\cE\cM_\cE \>=\> a\cM_\cE\>.
\]
Now we apply Proposition~\ref{ann}.
\begin{proposition}
Let $\cE$ be an Artin-Schreier extension or a Kummer extension of degree $p=\chara Kv$.
Assume that $[L:K]=(vL:vK)$ and that $\Omega_{\cO_L|\cO_K}\ne 0$. Then
\[
\ann \Omega_{\cO_L|\cO_K} \>=\> \left\{
\begin{array}{ll}
(a\cM_\cE)^{p-1} & \mbox{if $\cM_\cE$ is a principal $\cO_\cE$-ideal, } \\
(a\cO_\cE)^{p-1} = a^{p-1}\cO_\cE & \mbox{if $\cM_\cE$ is a nonprincipal
$\cO_\cE$-ideal.}
\end{array}\right.
\]
Further, $\ann \Omega_{\cO_L|\cO_K}=\cM_L$ if and only if $p=2$, $a\notin\cM_L=
\cM_\cE$ and $\cM_L$ is a principal $\cO_L$-ideal.
\end{proposition}

Let us note that if (DRvp) holds (and in particular, if $(K,v)$ is a deeply ramified
field), then the maximal ideal of any coarsening of $\cO_L$ is never principal. In this
case, $\cM_L$ is never the annihilator of $\Omega_{\cO_L|\cO_K}$.

\pars
In the case of a Kummer extension of prime degree $q=(vL:vK)\ne \chara Kv$,
(\ref{OKumec1}) holds, and we set $a=1$. Then we obtain from Theorem~\ref{OKume}
and Proposition~\ref{ann}:
\begin{proposition}
Let $\cE$ be a Kummer extension of degree $q=(vL:vK)\ne\chara Kv$.
Assume that $\Omega_{\cO_L|\cO_K}$ is nonzero. Then $\cM_\cE$ is a principal
$\cO_\cE$-ideal, and
\[
\ann \Omega_{\cO_L|\cO_K} \>=\> \cM_\cE^{q-1}\>.
\]
Further, $\ann \Omega_{\cO_L|\cO_K}=\cM_L$ if and only if $q=2$ and $\cM_\cE=\cM_L$.
\end{proposition}

\mn
%
%
\section{Deeply ramified fields in equal characteristic with prescribed associated
convex subgroups}               \label{sectpacse}
%

%
%
\subsection{Preliminaries from ramification theory}                               
An algebraic extension $(L|K,v)$ of a henselian valued field
$(K,v)$ is called \bfind{tame} if every finite subextension $K'|K$ satisfies the
following conditions:
\sn
(T1) \ the ramification index $(vK':vK)$ is not divisible by $\chara Kv$,
\sn
(T2) \ the residue field extension $K'v|Kv$ is separable,
\sn
(T3) \ the extension $(K'|K,v)$ is defectless.
\sn
A henselian valued field $(K,v)$ is called a \bfind{tame field} if the algebraic
closure $K^{\rm ac}$ of $K$ with the unique extension of $v$ is a tame extension of
$(K,v)$. It follows from conditions (T1)--(T3) that all tame fields are perfect
defectless fields. For the algebra and model theory of tame fields, see \cite{Ku39}.

\pars
The \bfind{ramification field} of a Galois extension $(L|K,v)$ with Galois group
$G=\Gal(L|K)$ is the fixed field in $L$ of the \bfind{ramification group}
\begin{equation}
G^r\>:=\> \left\{\sigma\in G\>\left|\;\> \frac{\sigma b -b}{b}\in \cM_L
\mbox{ \ for all }b\in L^\times\right.\right\}\>.
\end{equation}
When dealing with a valued field $(K,v)$, we will tacitly assume $v$ extended to its
algebraic closure. Then the \bfind{absolute ramification field of $(K,v)$} (with
respect to the chosen extension of $v$), denoted by $(K^r,v)$, is the
ramification field of the Galois extension $(K\sep|K,v)$.
If $(K(a)|K,v)$ is finite and a defect extension, then $(K^r(a)|K^r,v)$ is a
defect extension with the same defect (see \cite[Proposition~2.13]{KR}). On the
other hand, $K\sep|K^r$ is a $p$-extension (see \cite[Theorem (20.18)]{End}), so
every finite extension of $K^r$ is a tower of purely inseparable extensions and
Galois extensions of degree $p$. If $(K,v)$ is henselian, then $(K^r,v)$ is its
unique maximal tame extension (see \cite[Proposition 4.1]{Ku2}). Hence the next fact
follows from \cite[Proposition~2.13]{KR}:
\begin{lemma}                               \label{defthrutame}
If $(K,v)$ is henselian, $(K(a)|K,v)$ is finite and a defect extension, and $(L|K,v)$
is a tame extension, then $d(L(a)|L,v)=d(K(a)|K,v)$.
\end{lemma}

An extension $(L,v)$ of a henselian field $(K,v)$ is called \bfind{purely wild} if
every finite subextension $(L_0|K,v)$ satisfies:
\sn
(a) $(vL_0:vK)$ is a power of the characteristic exponent of $Kv$,
\sn
(b) $L_0v|Kv$ is purely inseparable.
\sn
The extension $(L,v)$ of $(K,v)$ is purely wild if and only if it is linearly disjoint
from $K^r$ over $K$ (see \cite[Lemma 4.2]{Ku2}).
\begin{lemma}                              \label{mpwext}
Every maximal purely wild extension of a henselian field is a tame field.
\end{lemma}
\begin{proof}
By \cite[Theorem 4.3]{Ku2}, every maximal purely wild extension $W$ of a henselian
field $(K,v)$ is a $K$-complement of $K^r$, that is, $W\cap K^r=K$ and $W.K^r=
K^{\rm ac}$. By \cite[Lemma 2.1 (i)]{Ku2}, there is also a $W$-complement $W'$ of
$W^r$. Again by \cite[Theorem 4.3]{Ku2}, $W'$ is a maximal purely wild extension of
$W$. By the maximality of $W$, we must have $W'=W$. Hence $K^{\rm ac}=W'.W^r=W^r$,
which shows that $W$ is a tame field.
\end{proof}

%
%
\subsection{Technical preliminaries}                               
For the following result, see~\cite[Lemma 4.1]{[Bl3]}
(cf.\ also~\cite[Lemma 2.21]{Ku30}):
\begin{lemma}                                   \label{imm_deg_p}
Assume that $(K(a)|K,v)$ is a unibranched extension of prime degree such
that $v(a-K)$ has no maximal element. Then the
extension $(K(a)|K,v)$ is immediate and hence a defect extension.
\end{lemma}

\begin{lemma}                 \label{l1}
1) Let $(K_0,v)$ be a valued field of characteristic $p>0$ whose value group
is not $p$-divisible. Take $a\in K_0$ such that $va<0$ is not divisible by
$p$. Let $\vartheta$ be a root of the Artin--Schreier polynomial $X^p-X-a$. Then
$(K_0^{1/p^{\infty}}(\vartheta)|K_0^{1/p^{\infty}},v)$ is a defect extension with
independent defect, and $v(\vartheta-K_0^{1/p^{\infty}})\subseteq
(vK_0^{1/p^{\infty}})^{<0}$.
\sn
2) Take a perfect field $k$ of characteristic $p>0$ and $K_0$ to be $k(t)$, $k(t)^h$
or $k((t))$, equipped with the $t$-adic valuation $v=v_t\,$. Let $\vartheta$ be
a root of the Artin--Schreier polynomial $X^p-X-1/t$. Then the assertion of part 1)
holds, and
\[
v(\vartheta-K_0^{1/p^{\infty}})\>=\>(v K_0^{1/p^{\infty}})^{<0}\;.
\]
\end{lemma}
\begin{proof}
1): We have that
$v\vartheta=va/p$ and $[K_0(\vartheta):K_0]=p= (vK_0(\vartheta):vK_0)$. The Fundamental
Inequality (cf.\ (17.5) of \cite{End} or Theorem 19 on p.~55 of \cite{ZS} shows
that $K_0(\vartheta)v=K_0v$ and that the extension $(K_0(\vartheta)|
K_0,v)$ is unibranched. The further extension of $v$ to the perfect hull
\[
K_0(\vartheta)^{1/p^{\infty}}= K_0^{1/p^{\infty}}(\vartheta)
\]
is unique, as the extension is purely inseparable. It follows that
also the extension $(K_0^{1/p^{\infty}}(\vartheta)|K_0^{1/p^{\infty}},v)$ is
unibranched. On the other hand, $[K_0^{1/p^{\infty}} (\vartheta):K_0^{1/p^{\infty}}]
=p$ since the separable extension $K_0(\vartheta)|K_0$ is linearly disjoint from
$K_0^{1/p^{\infty}}|K_0$. The value group $vK_0^{1/p^{\infty}}(\vartheta)=
vK_0(\vartheta)^{1/p^{\infty}}$ is the $p$-divisible hull of $vK_0(\vartheta)=
vK_0+\Z v\vartheta$. Since $pv\vartheta\in vK$, this
is the same as the $p$-divisible hull of $vK_0$, which in turn is equal to
$vK_0^{1/p^{\infty}}$. The residue field of $K_0^{1/p^{\infty}}(\vartheta)$ is the
perfect hull of $K_0(\vartheta)v=K_0v$. Hence it is equal to the residue field of
$K_0^{1/p^{\infty}}$. It follows that the extension $(K_0^{1/p^{\infty}}(\vartheta)|
K_0^{1/p^{\infty}},v)$ is immediate and that its defect is $p$, equal to its degree.
Since $K_0^{1/p^{\infty}}$ is perfect, it is deeply ramified and hence according to
\cite[part (1) of Theorem~1.10]{KR} the extension must
have independent defect. The inclusion $v(\vartheta-K_0^{1/p^{\infty}})\subseteq
(vK)^{<0}$ follows from \cite[Corollary~2.30]{Ku30}.

\mn
2): In all three cases we have that $K_0^{1/p^{\infty}}=K_0(t^{1/p^k}\mid k\in\N)$.
For the partial sums
\begin{equation}
b_k\>:=\>\sum_{i=1}^k t^{-1/p^i}\;\in\> K_0^{1/p^{\infty}}
\end{equation}
we have
\[
(\vartheta-b_k)^p-(\vartheta-b_k)\>=\> \vartheta^p-\vartheta-b_k^p+b_k \>=\>
\frac 1 t -\sum_{i=0}^{k-1} t^{-1/p^i} + \sum_{i=1}^k t^{-1/p^i} \>=\> t^{-1/p^k}\>,
\]
so
\[
v (\vartheta-b_k) \>=\> -\frac 1 {p^{k+1}}\><\>0 \>.
\]
Suppose that there is $c\in K_0^{1/p^{\infty}}$ such that $v(\vartheta-c)>-1/p^k$ for
all $k$. Then $v(c-b_k)=\min\{v(\vartheta-c),v(\vartheta-b_k)\}=-1/p^{k+1}$ for all
$k$. On the other hand, there is some $k$ such that $c\in K_0(t^{-1/p},\ldots,
t^{-1/p^k})=K_0(t^{-1/p^k})$. But this contradicts the fact that $v(c-t^{-1/p}-\ldots
-t^{-1/p^k})=v(c-b_k)= -1/p^{k+1}\notin vK_0(t^{-1/p^k})$. As $v(\vartheta
-K_0^{1/p^{\infty}})\subseteq (vK_0^{1/p^{\infty}})^{<0}$ by part 1), this proves
that the values $-1/p^k$ are cofinal in $v(\vartheta-K_0^{1/p^{\infty}})$. Since
$vK_0^{1/p^{\infty}}$ is a subgroup of the rationals, this shows that the least
upper bound of $v(\vartheta-K_0^{1/p^{\infty}})$ in $vK_0^{1/p^{\infty}}$ is the
element 0. As $v(\vartheta-K_0^{1/p^{\infty}})$ is an initial segment of
$vK_0^{1/p^{\infty}}$ by \cite[Lemma~2.19]{Ku30}, we conclude that
$v(\vartheta-K_0^{1/p^{\infty}})=(vK_0^{1/p^{\infty}})^{<0}$.
\end{proof}

When we take $K_0=\F_p((t))$ in part 1) of this lemma, where $\F_p$ is the field with
$p$ elements, and $a=1/t$, we obtain ``Abhyankar's Example'', see
\cite[Example~3.12]{Ku31}.

\pars
\begin{lemma}                 \label{l2}
Take a valued field $(K,v)$ of characteristic $p>0$, a decomposition $v=w\circ\bar w$,
and an Artin-Schreier extension of $K$ with Artin-Schreier generator $\vartheta$.
Then the following assertions hold.
\sn
1) $(K(\vartheta)|K,v)$ is a defect extension
with $v(\vartheta-K)=\{\alpha\in vK(\vartheta)\mid \alpha<\bar w(K(\vartheta)w)\}$
if and only if $(K(\vartheta)|K,w)$ is a defect extension with $w(\vartheta-K)=
(wK(\vartheta))^{<0}$.
\sn
2) If $w\vartheta=0$ and $\bar w(\vartheta w-Kw)=(\bar w(K(\vartheta)w))^{<0}$, then
$(K(\vartheta)|K,v)$ is a defect extension with $v(\vartheta-K)=(vK(\vartheta))^{<0}$.
\end{lemma}
\begin{proof}
1): We can write $wK(\vartheta)=vK(\vartheta)/\bar w(K(\vartheta)w)$ and $wa=va+\bar w
(K(\vartheta)w)$ for each $a\in K(\vartheta)$. This implies that
$v(\vartheta-K)=\{\alpha\in vK\mid \alpha<\bar w(K(\vartheta)w)$ if and only
if $w(\vartheta-K)=(wK(\vartheta))^{<0}$.
By Lemma~\ref{imm_deg_p},
$(K(\vartheta)|K,v)$ is a defect extension if $v(\vartheta-K)$ is a subset of
$vK(\vartheta)$ without maximal element, and similarly for $w$ in place of $v$. This
fact together with the equivalence we have already shown proves the assertion of our
lemma.
\sn
2): Our assumption implies that $v(\vartheta-K)\subseteq (vK(\vartheta))^{<0}$
since if there is $c\in K$ such that $v(\vartheta-c)\geq 0$, then $vc=
v\vartheta$, so $wc=w\vartheta=0$, and $\bar w(\vartheta w-cw)\geq 0$. On the other
hand, $\bar w(Kw)$ is a convex subgroup of $vK$, so $(\bar w(Kw))^{<0}$ and thus also
$\bar w(\vartheta w-Kw)$ is cofinal in $vK^{<0}$. Since for every $b\in Kw$ there is
$c\in K$ with $cw=b$ and $v(\vartheta-c)=\bar w(\vartheta w-cw)$, it follows that
$v(\vartheta-K)=(vK(\vartheta))^{<0}$.
Again by Lemma~\ref{imm_deg_p}, $(K(\vartheta)|K,v)$ is a defect extension.
\end{proof}

\begin{proposition}                  \label{prepmixcase}
Take any perfect field $K$ of characteristic $p>0$ and a field $L_1$ of characteristic
$0$ carrying a $p$-adic valuation $v_p$ such that $v_p L_1=\Z v_p p$ and $L_1v_p=K$.
Set $a_0:=p$ and by induction, choose elements $a_i$ in the algebraic closure
$L_1^{\rm ac}$ of $L_1$ such that $a_i^p=a_{i-1}$ for $i\in\N$, and set
\[
L_2:=L_1(a_i\mid i\in\N)\>.
\]
Further, take $a\in \Q^{\rm ac}\subseteq L_2^{\rm ac}$ such that
\begin{equation}                           \label{a}
a^p\,-\,a\>=\>\frac 1 p\>.
\end{equation}
Then the following assertions hold.
\sn
1) There is a unique extension of the valuation $v_p$ to $L_2$.
\sn
2) $(L_2,v_p)$ is a deeply ramified field with value group $v_pL_2=
\frac{1}{p^{\infty}}\Z v_p p$ and residue field $L_2v_p=K$.
\sn
3) $(L_2(a)|L_2,v_p)$ is a defect extension of degree $p$.
\sn
4) Assume that $\F_p^{\rm ac}\subseteq K$ and there is $L_0\subseteq L_1$ such that
$(L_0,v_p)$ is henselian with $L_0v_p=\F_p^{\rm ac}$. Then there is a finite extension
$(L|L_2,v_p)$ such that $Lv_p=L_2v_p=K$, $(L,v_p)$ is a deeply ramified field, and
$(L(a)|L,v_p)$ is a Galois defect extension of degree $p$ with independent
defect and associated convex subgroup $\{0\}$.
\end{proposition}
\begin{proof}
By our choice of the $a_i\,$, $\frac{v_pp}{p^i}=v_pa_i\in v_pL_1(a_i)$. Therefore,
\[
p^i\>\leq\> (v_pL_1(a_{i}):v_pL_1) \>\leq\> (v_pL_1(a_{i}):v_pL_1)[L_1(a_{i})v_p:
L_1 v_p] \>\leq\> [L_1(a_{i}):L_1] \>\leq\> p^i\>.
\]
Hence equality holds everywhere, and $[L_1(a_{i})v_p:L_1 v_p]=1$. We thus obtain that
$v_pL_1(a_i)=\frac{1}{p^i}v_pL_1$
and $L_1(a_i)v_p=L_1 v_p$. Consequently,
\[
v_pL_2\>=\>\bigcup_{i\in\N} v_pL_1(a_i)\>=\>\frac{1}{p^{\infty}}\Z \quad\mbox{and}
\quad L_2v_p\>=\>L_1 v_p\>=\>K\>,
\]
and the extension $(L_2|L_1,v_p)$ is unibranched, which proves assertion 1). We see
that $v_pL_2$ is $p$-divisible, so $(L_2,v_p)$ satisfies (DRvg).
In order to show that $(L_2,v_p)$ is a deeply ramified field it
remains to show that it satisfies (DRvr).

Take $b\in \cO_{L_2}\,$. Then $b\in L_1(a_i)$ for some $i\in\N$ and we can write
\[
b \>=\> \sum_{j=0}^{p^i-1} c_j a_i^j\>
\]
with $c_j\in L_1\,$. Since the values $v_pa_i^j$, $0\leq j\leq p^i-1$ lie in distinct
cosets modulo $v_pL_1$ (hence the elements $a_i^j$, $0\leq j\leq p^i-1$ form a
valuation basis of $(L_1(a_i)|L_1,v_p)$), we have that $v_pb=\min_{0\leq j\leq p^i-1}
v_p c_j a_i^j$. As $b\in \cO_{L_2}\,$, it follows that $c_j a_i^j\in \cO_{L_2}$ for all $j$.
We observe that $v_p a_i^j \leq \frac{p^i-1}{p^i}v_p p <v_p p$, so $v_p c_j\in \Z
v_p p$ cannot be negative. This shows that $c_j\in \cO_{L_1}$ for all $j$.

For every $c\in\cO_{L_1}$ there is $d\in\cO_{L_1}$ such that $c\equiv d^p \mod
p\cO_{L_1}\,$; indeed, as $K$ is perfect, there is $\xi\in K$ such that $cv_p=\xi^p$,
so we can choose $d\in\cO_{L_1}$ such that $dv_p=\xi$. Then we obtain
$v_p(c-d)\geq v_pp$.

For each $j$ we now choose $d_j\in L_1$ such that $c_j\equiv d_j^p \mod p\cO_{L_1}\,$.
Then
\[
\left(\sum_{j=0}^{p^i-1} d_j a_{i+1}^j\right)^p \>\equiv\> \sum_{j=0}^{p^i-1} d_j^p
(a_{i+1}^p)^j \>\equiv\>\sum_{j=0}^{p^i-1} c_j a_i^j \>=\> b \mod p\cO_{L_1(a_i)}\>.
\]
In view of \cite[Lemma~4.1 (2)]{KR}, this shows that $(K,v_p)$ satisfies (DRvr) and
is therefore a deeply ramified field, which proves assertion 2).

\parm
Our next aim is to show that the extension
$(L_2(a)|L_2,v_p)$ is nontrivial and immediate. For each $i\in\N$, we set
\[
b_i \>=\> \sum_{j=1}^i \frac 1 {a_j}  \>\in\> L_1(a_i)
\]
and compute, using \cite[Lemma~2.17 (2)]{KR}:
\begin{eqnarray*}
(a-b_i)^p - (a-b_i) &\equiv& a^p - \sum_{j=1}^i \frac 1 {a_j^p} - a +
\sum_{j=1}^i \frac 1 {a_j} \\
&=& \frac 1 p  - \frac 1 p -\sum_{j=1}^{i-1} \frac 1 {a_j}+\sum_{j=1}^i \frac 1 {a_j}
\>=\> \frac 1 {a_i}  \mod \cO_{L_1(a_i)}\>.
\end{eqnarray*}
It follows that $v_p(a-b_i)<0$ and
\[
-\frac{v_pp}{p^i}\>=\> v_p\frac{1}{a_i} \>=\> \min \{pv_p(a-b_i),v_p(a-b_i)\}
\>=\> pv_p(a-b_i)\>,
\]
whence
\begin{equation}                                  \label{varth-bi}
v_p(a-b_i) \>=\> -\frac{v_pp}{p^{i+1}}\>.
\end{equation}
We have that
\begin{eqnarray*}
p &\leq& (v_pL_1(a_i,a):v_pL_1(a_i)) \>\leq\> (v_pL_1(a_i,a):v_pL_1(a_i))
[L_1(a_i,a)v_p:
L_1(a_i)v_p] \\
&\leq& [L_1(a_i,a):L_1] \>\leq\> p\>.
\end{eqnarray*}
Thus equality holds everywhere and we have that $(v_pL_1(a_i,a):v_pL_1(a_i))=p$, the
extension is unibranched, $L_1(a_i,a)v_p=L_1(a_i)v_p=L_1 v_p$, and for all $i\in\N$,
$a\notin L_1(a_i)$. Hence $a\notin L_2$, and we have:
\[
v_pL_2(a)\>=\>\bigcup_{i\in\N} v_pL_1(a_i,a)\>=\>\frac{1}{p^{\infty}}\Z
\>=\> v_pL_2 \quad\mbox{and}\quad L_2(a)v_p\>=\>L_1 v_p \>=\> L_2v_p\>.
\]
This shows that $(L_2(a)|L_2,v_p)$ is nontrivial and immediate, as asserted.
The extension is also unibranched since each extension $(L_1(a_i,a)|L_1(a_i),v_p)$
is unibranched. Therefore, it is a defect extension of degree $p$, which proves
assertion 3).

\pars
Since $[L_0(a):L_0]=p$, there is an element $\zeta'\in L_0^{\rm ac}$ such that
$[L_0(\zeta'):L_0]$ divides $(p-1)!$ and $L_0(a,\zeta')|L_0(\zeta')$ is Galois. As
$(L_0,v_p)$ is henselian, $p$ does not divide $[L_0(\zeta'):L_0]$, and $L_0v_p$ is
algebraically closed, the Lemma of Ostrowski shows that $[L_0(\zeta'):L_0]=
(v_pL_0(\zeta'):v_pL_0)$. It follows that also $[L_1(\zeta'):L_1]=(v_pL_1(\zeta'):
v_pL_1)$, and that $(L_1(\zeta')|L_1,v_p)$ is a tame extension. Hence by
Lemma~\ref{defthrutame}, also $(L_1(a,\zeta')|L_1(\zeta'),v_p)$ is a defect
extension. By \cite[Theorem 1.5]{KR},
the algebraic extension $(L_1(\zeta'),v_p)$ of $(L_1,v_p)$ is again a deeply
ramified field and hence an rdr field. Thus it follows from Theorem 1.2 that the Galois
extension $(L_1(a,\zeta')|L_1(\zeta'),v_p)$ has independent defect. Since $(L_2,v_p)$
has rank $1$, the convex subgroup associated with the extension $(L_2(\zeta',a)|L_2
(\zeta'),v_p)$ is $\{0\}$. With $L:=L_2(\zeta')$, we have now proved assertion 4).
\end{proof}

%
%
\subsection{The case of equal characteristic: some examples}     
\begin{example}                  \label{e1}
Take $K_0=\F_p((t))$. Then $K:=K_0^{1/p^\infty}=\F_p((t))^{1/p^\infty}$ is a perfect
field, so under the canonical $t$-adic valuation $v_t$ it is a deeply ramified
field. Moreover, $(K,v_t)$ is henselian and of rank $1$. Now set $L_0:=K((x))$ and
$L:=L_0^{1/p^\infty}$, and equip $L$ with the canonical $x$-adic valuation $v_x\,$.
Both $v_x$ and $v_t$ are henselian, hence so is their composition $v:=v_x\circ v_t$ on
$L$. As $L$ is perfect, $(L,v)$ is a deeply ramified field. Its value group
$vL$ has rank $2$, i.e., it has two proper convex subgroups. Let $\vartheta_t$ be a
root of $X^p-X-\frac{1}{t}$, and
$\vartheta_x$ a root of $X^p-X-\frac{1}{x}$. We note that $K=Lv_x$.

By part 1) of Lemma~\ref{l1}, both extensions $(L(\vartheta_x)|L,v)$ and
$(L(\vartheta_x)|L,v_x)$ are defect extensions with independent defect. By part 2) of
Lemma~\ref{l1}, $v_x(\vartheta_x-L)=v_xL^{<0}$, which by part 1) of Lemma~\ref{l2}
implies that $v(\vartheta_x-L)=\{\alpha\in vL\mid\alpha<v_tK\}$. Hence for $\cE_x=
(L(\vartheta_x)|L,v)$, its associated convex subgroup $H_{\cE_x}$ is the convex
subgroup $v_tK$ of $vL$.

Again by part 1) of Lemma~\ref{l1}, the extension $(K(\vartheta_t)|K,v_t)$ has
independent defect, and by part 2) of Lemma~\ref{l1}, $v_t(\vartheta_t-K)=
v_t K^{<0}$. By part 2) of Lemma~\ref{l2}, $(L(\vartheta_t)|L,v)$ is a defect
extension with $v(\vartheta_t-L)=(vL(\vartheta_t)^{<0}$. Therefore, $\{0\}$ is the
convex subgroup associated with the defect extension $(L(\vartheta_t)|L,v)$.

We have shown that both convex subgroups of $vL$, $v_tK$ and $\{0\}$, appear as the
convex subgroups associated with Galois defect extensions of $(L,v)$.
\end{example}

Let us present a modification of this example.
\begin{example}                  \label{e2}
In the previous example, we replace $K$ by some algebraically closed (or just
henselian defectless) field with an
arbitrary nontrivial valuation $v_t\,$. Then $(K,v_t)$ has no defect extensions, and
$H=v_tK$ will be the only convex subgroup of $vL$ associated with Galois defect
extensions. This is seen as follows. As in Example~\ref{e1} we have that
$(L(\vartheta_x)|L,v)$ is a defect extension with $v(\vartheta_x-L)=\{\alpha\in vL\mid
\alpha<v_tK\}$. On the other hand, suppose that there is a defect
extension $(L(\vartheta)|L,v)$ with $v(\vartheta-L)=vL^{<0}$. Then there is $b\in
L$ such that $v(\vartheta-b)\in v_tK$. Set $\bar\vartheta:=(\vartheta-b)v_x\,$. Since
$\vartheta-b$ is a root of an Artin-Schreier polynomial over $L$, $\bar\vartheta$ is
a root of an Artin-Schreier polynomial over
$K$. By construction, $(K(\bar\vartheta)|K,v_t)$ cannot be a defect extension,
so there is $\bar c\in K$ such that $v_t(\bar\vartheta-\bar c)$ is the maximum of
$v_t(\bar\vartheta-K)$. Choose $c\in L$ such that $cv_x=\bar c$. Then $v(\vartheta
-b-c)$ is the maximum of $v(\vartheta-L)$, contradicting our assumption that
$(L(\vartheta)|L,v)$ is a defect extension.

\pars
At the other extreme, we can keep $(K,v_t)$ and $(L_0,v_x)$ as in the previous example,
but now take $(L,v_x)$ to be a maximal purely wild extension of $(L_0,v_x)$. As $L_0
v_x=K$ is perfect, we have $Lv_x=L_0v_x=K$. By Lemma~\ref{mpwext}, $(L,v_x)$ is a tame
field and thus has
no defect extensions. Therefore $v_tK$ cannot appear as a convex subgroup associated
with any Galois defect extension; indeed, if $v(\vartheta-L)=\{\alpha\in vL\mid
\alpha<v_tK\}$, then by part 1) of Lemma~\ref{l2}, $(L(\vartheta)|L,v_x)$
would be a defect extension. However, as in the previous example one shows that
$\cE_t=(L(\vartheta_t)|L,v)$ is a defect extension with $H_{\cE_t}=\{0\}$.
\end{example}

%
%
\subsection{The case of equal characteristic: a general construction}  
We will now present a much more general construction.
Given any ordered index set $I$ and for every $i\in I$ an arbitrary ordered abelian
group $C_i\,$, we can form the \bfind{Hahn sum} $\coprod_{i\in I}C_i\,$. As an abelian
group, this is the direct sum of the groups $C_i\,$, represented as the set of all
tuples $(\alpha_i)_{i\in I}\,$ with only finitely many of the $\alpha_i\in C_i$
nonzero. An ordering on $\coprod_{i\in I}C_i$ is introduced as follows. For
$(\alpha_i)_{i\in I}\in \coprod_{i\in I}C_i\,$, set $i_{\rm min}:=
\min\{i\in I\mid \alpha_i\ne 0\}$. Then define $(\alpha_i)_{i\in I}>0$ if
$\alpha_{i_{\rm min}}>0$. If all $C_i$ are archimedean ordered, then the principal
convex subgroups are exactly the subsets of the form $\{(\alpha_i)_{i\in I}\in
\coprod_{i\in I}C_i \mid \alpha_i=0 \mbox{ for all } i<i_0\}$ for some $i_0\in I$;
this subgroup is generated by any $(\alpha_i)_{i\in I}\in \coprod_{i\in I}C_i$ with
$\alpha_{i_0}\ne 0$; likewise, the subprincipal convex subgroups are exactly the
subsets of the form $\{(\alpha_i)_{i\in I}\in \coprod_{i\in I}C_i \mid \alpha_i=0
\mbox{ for all } i\leq i_0\}$ for some $i_0\in I$.

Now take any ordered index set $I$. Set $C_i=\Z$ for all $I$ and let $\Gamma_0$ be
the Hahn sum $\coprod_{i\in I}C_i\,$. For each $\ell\in I$ let
$1_\ell$ denote the element $(\alpha_i)_{i\in I}$ with $\alpha_i=1\in\Z$ if $i=\ell$
and $\alpha_i=0$ otherwise. Now the elements $1_\ell$ generate all principal convex
subgroups of $\Gamma_0\,$. Note that if $\ell<\ell'$, then $1_\ell\gg 1_{\ell'}\,$,
that is, $1_\ell>n 1_{\ell'}$ for all $n\in\N$.

Take a perfect field $k$ of characteristic $p>0$ and a set $\{t_i\mid i\in I\}$ of
elements algebraically independent over $k$ and define a valuation $v$ on the field
$k(t_i\mid i\in I)$ by setting $vt_i=1_i$ for each $i\in I$. Let $(K_0,v)$ be the henselization of $(k(t_i\mid i\in I),v)$.

For each $i\in I$ there are:
\sn
$\bullet$ \ a
decomposition $v=v_i\circ\bar v_i\,$, where $v_i$ is the finest coarsening
of $v$ on $K_0$ that is trivial on $t_i$ and $\bar v_i$ is the valuation induced by
$v$ on the residue field $K_0 v_i\,$, which can be identified with $k(t_j\mid
i\leq j\in I)$, and
\sn
$\bullet$ \ a decomposition $\bar v_i=w_i\circ\bar w_i\,$, where $w_i$ is the
$t_i$-adic valuation on $K_0 v_i$ and $\bar w_i$ is the valuation induced by
$\bar v_i$ on the residue field $K_0 v_iw_i\,$, which can be identified with
$k(t_j\mid i< j\in I)$.
\sn
Note that $v_j$ is strictly coarser than $v_i$ if $j<i$.

\pars
We take $K_1$ to be the perfect hull of
$K_0$, that is, $K_1=k(t_i^{1/p^n}\mid i\in I,\,n\in\N)$. The valuations $v$ and
$v_i\,$, $i\in I$, have unique extensions to $K_1\,$, and $vK_1$ is the $p$-divisible
hull $\frac 1 {p^\infty} \Gamma$ of $vK_0\,$. Further, $K_1 v_i$ is
the perfect hull $k(t_j^{1/p^n}\mid i\leq j\in I,\,n\in\N)$ of
$k(t_j\mid i\leq j\in I)$, so the valuations $\bar v_i$ and $w_i$ have unique
extensions to $K_1 v_i$ and the decompositions $v=v_i\circ \bar v_i$ again hold on
$K_1\,$. Likewise, $K_1 v_i w_i$ is the perfect hull $k(t_j^{1/p^n}\mid i< j\in I,
\,n\in\N)$ of
$k(t_j\mid i< j\in I)$, so also the valuations $\bar v_i$ have unique extensions to
$K_1 v_i w_i$ and the decompositions $v_i=w_i\circ \bar w_i$ again hold on $K_1\,$.

We set $\Gamma:=vK_1=\frac 1 {p^\infty} \Gamma_0$ and define $H_i$ to be the largest
convex subgroup of $\Gamma$ that does not contain $1_i\,$, that is, $H_i=
\bar w_i(K_1 v_i w_i)$. Consequently, the $H_i$ are exactly all
subprincipal convex subgroups of $\Gamma$. The principal convex subgroups of $\Gamma$
are exactly all smallest convex subgroups that contain $1_i$ for some $i\in I$;
they are of the form $\bar v_i(K_1 v_i)$.

\pars
The next theorem proves part 1) of Theorem~\ref{prescracs}.
\begin{theorem}                             \label{t1}
Take any subset $J\subseteq I$. Then there exists an algebraic extension $(K_2,v)$
of $(K_1,v)$ which is a henselian deeply ramified
field and such that the convex subgroups associated with Galois defect extensions of
prime degree of $(K_2,v)$ are exactly the convex subgroups $H_j$ with $j\in J$.
\end{theorem}
\begin{proof}
Since $(K_1,v)$ is henselian and perfect, each algebraic extension $(K,v)$ of
$(K_1,v)$ is a henselian deeply ramified field.

For each $i\in I$ we let $\vartheta_i$ be a root of $X^p-X-\frac{1}{t_i}$. Since $v_i$ is
trivial on $t_i\,$, we can identify $\vartheta_i$ with $\vartheta_i v_i\,$. By part 1)
of Lemma~\ref{l1}, $(K_1(\vartheta_i)|K_1,v)$, $(K_1v_i(\vartheta_i)|K_1 v_i,\bar v_i)$
and $(K_1v_i(\vartheta_i)|K_1 v_i,w_i)$ are defect extensions with independent defect.
By construction, $K_1 v_i=K_1 v_i w_i(t_i^{1/p^n}\mid n\in\N)=K_1 v_i w_i
(t_i)^{1/p^{\infty}}$, where $K_1 v_i w_i$
is a perfect field. Hence by part 2) of Lemma~\ref{l1}, $w_i(\vartheta_i-K_1 v_i)=
(w_i (K_1 v_i))^{<0}$, which by part 1) of Lemma~\ref{l2} implies that $\bar v_i
(\vartheta_i-K_1v_i)=\{\alpha\in \bar v_i(K_1v_i)\mid\alpha<\bar v_i(K_1v_iw_i)\}=
\{\alpha\in \bar v_i(K_1v_i)\mid\alpha<H_i\}$. We claim that this implies that
$v(\vartheta_i-K_1)=
\{\alpha\in vK_1\mid\alpha<H_i\}$, that is, $H_i$ is the convex subgroup associated
with the defect extension $(K_1(\vartheta_i)|K_1,v)$. For the proof of the claim,
observe that $K_1v_i\subset K_1$ and $v|_{K_1v_i}=\bar v_1\,$; so
$\bar v_i(\vartheta_i-K_1v_i)\subseteq v(\vartheta_i-K_1)$. We show that the former is
cofinal in the latter, which will prove our claim. Take any $c\in K_1$. If $vc<
v\vartheta_i\,$, then $\bar v_i\vartheta_i=v\vartheta_i>v(\vartheta_i-c)$. If $vc\geq
v\vartheta_i\,$, then we can write $c=cv_i+c'$ with $cv_i\in K_1 v_i$ and $c'\in
K_1$ with $v_i c'>0$. It follows that $vc'>\bar v_i(K_1 v_i)$ and consequently,
$vc'>\bar v_i(\vartheta_i-cv_i)$ and $v(\vartheta_i-c)=v_i(\vartheta_i-cv_i)\in
\bar v_i(\vartheta_i-K_1v_i)$. Hence by
our construction, all $H_i$ for $i\in I$ appear as the convex subgroups associated
with Galois defect extensions of $(K_1,v)$. We now have to find an algebraic extension
of $(K_1,v)$ which will admit exactly all $H_i$ for $i\in I\setminus J$.

\pars
Let $(K_2,v)$ be a maximal algebraic extension of $(K_1,v)$ for which $vK_2=vK_1\,$,
the above decompositions carry over to $K_2$ for suitable extensions of the
valuations $v_i,\,\bar v_i,\, w_i,\, \bar v_i\,$, and for all $j\in J$, $K_2 v_j=
K_2v_jw_j(t_j)^{1/p^\infty}$. As $K_2v_jw_j$ is perfect, being an algebraic
extension of the perfect field $K_1v_jw_j$, part 2) of Lemma~\ref{l1} shows that
$(K_2v_j(\vartheta_j)|K_2v_j,w_j)$ is (still) a defect extension with $w_j(\vartheta_j
-K_2v_j)=(w_j(K_2v_j))^{<0}$. By part 2) of Lemma~\ref{l2},
$(K_2(\vartheta_j)|K_2,v_j\circ w_j)$ is a defect extension with $v_j\circ w_j
(\vartheta_j-K_2)=(v_j\circ w_j(K_2))^{<0}$. Now by part 1) of
Lemma~\ref{l2}, $(K_2(\vartheta_j)|K_2,v)$ is a defect extension with $v
(\vartheta_j-K_2)=\{\alpha\in vK_2\mid \alpha<\bar v_j(K_2 v_j w_j)\}$, that is,
$H_j=\bar v_j(K_1 v_j w_j)=\bar v_j(K_2 v_j w_j)$ is its associated convex subgroup.

\pars
Suppose that there is some $i\in I\setminus J$ such that $H_i$ is also the convex
subgroup associated with some Galois defect extension of $(K_2,v)$. In this case we
take $(L,w_i)$ to be a maximal purely wild extension $(L,w_i)$ of $(K_2v_i,w_i)$.
By Lemma~\ref{mpwext}, $(L,w_i)$ is a tame field and thus does not have any nontrivial
defect extensions. As $K_2v_i$ is perfect, being an algebraic extension of $K_1v_i\,$,
we have that $(L,w_i)$ is an immediate extension of $(K_2v_i,w_i)$, that is, $w_i L=
w_i(K_2v_i)$ and $Lw_i=K_2v_i w_i\,$. We take $(K_3,v_i)$ to be an algebraic
extension of $(K_2,v_i)$ such that $v_i K_3=v_i K_2$ and $K_3v_i=L$, and that
$[K'_2:K_2]=[K'_2 v_i:K_2 v_i]$ holds for every finite subextension $K'_2|K_2$ of
$K_3|K_2$; for the construction of such extensions, see \cite[Section 2.3]{Ku21}. We
set $v=v_i\circ w_i\circ\bar v_i$ on $K_3$; since $v_i K_3=v_i K_2$ and
$(K_3v_i,w_i)=(L,w_i)$ is an immediate extension of $(K_2v_i,w_i)$, also $(K_3,v)$
is an immediate extension of $(K_2,v)$.

Take any $j\in J$; we will show that we still have $K_3v_j=K_3 v_j w_j
(t_i)^{1/p^{\infty}}$. Since $K_3 v_i w_i$ is perfect, being an algebraic extension
of the perfect field $K_1 v_i w_i\,$, it will then follow as in the beginning of
this proof that $(K_3(\vartheta_j)|K_3,v)$ is still a defect
extension with associated convex subgroup $H_j\,$.

First assume that $j>i$. Then $K_3 v_j=K_2 v_j=K_2v_jw_j(t_j)^{1/p^\infty}=K_3v_jw_j
(t_j)^{1/p^\infty}$ since $K_3v_iw_i=Lw_i=K_2v_iw_i$ and $K_3 v_j$ and $K_3 v_jw_j$
are equal to or residue fields of $K_3v_iw_i\,$.
%

Now assume that $j<i$. Suppose that $K_3 v_j$ properly contains $K_3v_jw_j
(t_j)^{1/p^\infty}$. Then there is a finite subextension $K'_2|K_2$ of $K_3|K_2$
such that $K'_2 v_j$ properly contains $K'_2 v_jw_j(t_j)^{1/p^\infty}$. Using that
$K_2 v_j=K_2v_jw_j(t_j)^{1/p^\infty}$ and that $K'_2 v_i$ is equal to or a residue
field of $K'_2 v_jw_j$ and $K_2 v_i$ is equal to or a residue field of $K_2 v_jw_j\,$,
we compute:
\begin{eqnarray*}
[K'_2:K_2]&\geq& [K'_2 v_j:K_2 v_j]\> >\> [K'_2 v_jw_j(t_j)^{1/p^\infty}:K_2 v_jw_j
(t_j)^{1/p^\infty}]\\
&= & [K'_2 v_j w_j:K_2 v_j w_j]\>\geq\> [K'_2 v_i:K_2 v_i]\>=\>
[K'_2:K_2]\>.
\end{eqnarray*}
This contradiction proves that $K_3v_j=K_3 v_j w_j(t_i)^{1/p^{\infty}}$ also holds in
this case.

\pars
Finally, we show that $H_i$ cannot appear as the convex subgroup associated with
any Galois defect extension of $(K_3,v)$. This will contradict the maximality
of $(K_2,v)$ and show that it satisfies the statement of our theorem. Suppose the
contrary, and let $(K_3(\vartheta)|K_3,v)$ be an Artin-Schreier defect extension
with $H_i$ as its associated convex subgroup. Since $\bar v_i(K_3 v_i)$ properly
contains $H_i=\bar v_i(K_1 v_i w_i)$, it follows that there is some $b\in K_3$ such
that $v(\vartheta-b)\in \bar v_i(K_3 v_i)$. With $\vartheta':=\vartheta-b$ we
obtain that $\bar v_i(\vartheta'-K_3v_i)=\{\alpha\in \bar v_i(K_3v_i)\mid\alpha<
\bar v_i(K_1 v_i w_i)\}$. Hence Lemma~\ref{imm_deg_p} shows that
$(K_3v_i(\vartheta v_i)|K_3v_i,
\bar v_i)$ is a nontrivial defect extension. Thus by part 1) of Lemma~\ref{l2}, also
$(K_3v_i(\vartheta v_i)|K_3v_i,w_i)$ is a nontrivial defect extension. However, this
contradicts the fact that by our construction, $(K_3v_i,w_i)$ is a tame
and thus defectless field with respect to $w_i\,$.
\end{proof}

\mn
%
%
\subsection{The case of mixed characteristic}                
We choose a perfect field $K$ of characteristic $p>0$ containing $\F_p^{\rm ac}$.
We denote the $p$-adic valuation on $\Q$ by $v_p$ and take an algebraic extension
$(L_0,v_p)$ such that $(L_0,v_p)$ is henselian, $v_pL_0=v_p\Q$ and $L_0v_p=
\F_p^{\rm ac}$. Then we construct an extension $(L_1,v_p)$ of $(L_0,v_p)$ such that
$v_pL_1=v_p\Q$ and $L_1v_p=K$. See \cite[Section~2.3]{Ku21} for information on the
construction of such extensions. By Proposition~\ref{prepmixcase} there is an
algebraic extension $(L,v_p)$ of $(L_0,v_p)$ such that $Lv_p=K$ and $(L,v_p)$ is a
deeply ramified field admitting a Galois defect extension $(L(a)|L,v_p)$ of degree
$p$ with independent defect and associated convex subgroup $\{0\}$.

\begin{example}
Now take any nontrivial valuation $\bar v$ on $K$. If we choose $K$ to be
algebraically closed, then it does not admit any Galois defect extension. Still,
$(L(a)|L,v_p\circ\bar v)$ is a Galois defect extension, and as the convex subgroup
associated with the extension $(L(a)|L,v_p)$ is $\{0\}$, by part 1) of Lemma~\ref{l2}
the convex subgroup associated with the extension $(L(a)|L,v_p\circ\bar v)$ is
$\bar v K$. This is the only convex subgroup of $(v_p\circ\bar v)L$ that
appears as convex subgroup associated with some Galois defect extension of
$(L,v_p\circ\bar v)$.
\end{example}

The situation changes when $(K,\bar v)$ itself admits Galois defect extensions. Then
these can be lifted to Galois defect extensions of $(L,v_p\circ\bar v)$, and
the convex subgroups associated with Galois defect extensions of $(K,\bar v)$ appear
as convex subgroups associated with Galois defect extensions of $(L,v_p\circ
\bar v)$ that are properly contained in $\bar v(Lv_p)$. This will be exploited in the
\sn
{\it Proof of part 2) of Theorem~\ref{prescracs}.} Let $\Delta$ denote the largest
proper convex subgroup of $\Gamma$. Denote by $\cC_\Delta^{\rm sp}$ the set of all
proper convex subgroups of $\Delta$ in $\cC^{\rm sp}$. By part 1) of
Theorem~\ref{prescracs} we can choose a perfect henselian valued field $(K,\bar v)$
of characteristic $p$ for which the associated convex subgroups are exactly the
elements of $\cC_\Delta^{\rm sp}$. We take $(L,v_p)$ as
described at the beginning of this section and consider $(L,v_p\circ\bar v)$ which is
a deeply ramified field since $(L,v_p)$ is and $K$ is perfect.

Now $\bar vK$ is the largest proper convex subgroup of $(v_p\circ\bar v)L$ and it is
shown as in the proof of part 1) of Theorem~\ref{prescracs} that a convex subgroup of
$\bar v K$ is an associated convex subgroup for $(K,\bar v)$ if and only if it is
an associated convex subgroup for $(L,v_p\circ\bar v)$.

It remains to deal with the convex subgroup $\bar v K$ of $v_p\circ\bar v L$. If it
is an element of $\cC^{\rm sp}$, then we are done because $(L(a)|L,v_p)$ is a Galois
defect extension of degree $p$ with independent defect, and it follows that also
$(L(a)|L,v_p\circ\bar v)$ is a Galois defect extension of degree $p$ with
independent defect.

Finally, assume that $\bar v K$ is not an element of $\cC^{\rm sp}$. Then we replace
$(L,v_p)$ by a maximal purely wild extension, which does not change the residue field
$K$ because it is perfect, and is a tame field by Lemma~\ref{mpwext}. After this,
$(L,v_p)$ does not admit any defect extension
and $\bar v K$ cannot be an associated convex subgroup for $(L,v_p)$. It is then shown
as in the proof of part 1) of Theorem~\ref{prescracs} that it also cannot be an
associated convex subgroup for $(L,v_p\circ \bar v)$. This completes the proof of
part 2) of Theorem~\ref{prescracs}.   \qed

\bn

\end{document}